\documentclass[graybox,envcountsect,envcountresetsect, envcountsame]{svmult}
\usepackage{type1cm}        
\usepackage{makeidx}         
\makeindex		     
\usepackage{graphicx}        
\usepackage{multicol}        
\usepackage[bottom]{footmisc}

\usepackage{newtxtext}       %
\usepackage{newtxmath}       

\RequirePackage[all,cmtip]{xy}
\newdir^{ (}{{}*!/-5pt/\dir^{(}}
\newdir_{ (}{{}*!/-5pt/\dir_{(}}
\RequirePackage{xspace,url,bbm}

\usepackage[pagebackref,bookmarksopen]{hyperref}
\hypersetup{
	colorlinks=true,
	linkcolor=blue,
	citecolor=magenta,
	urlcolor=cyan,
}
\usepackage{bookmark} 

\makeatletter
\providecommand*{\toclevel@titlech}{0} 
\edef\toclevel@authorch{\the\numexpr\toclevel@titlech+1} 
\makeatother

\newenvironment{enumeratea}
{\bgroup 
\begin{enumerate}}
{\end{enumerate}\egroup}
\newenvironment{enumeratei}
{\bgroup \begin{enumerate}}
{\end{enumerate}\egroup}

\spnewtheorem{theoremeA}{Theorem}{\bfseries}{\itshape}

\spnewtheorem{theoremeAh}{Theorem}{\bfseries}{\itshape}

\newcommand\cf{see\kern.3em}
\newcommand\Cf{See\kern.3em}
\newcommand\eg{e.g.\kern.3em}
\newcommand\ie{i.e.,\ }
\newcommand\loccit{loc.\kern3pt cit.{}\xspace}
\newcommand\resp{\text{resp.}\kern.3em}
\newcommand\moins{\smallsetminus}

\newcommand\cswt{\widetilde}

\let\ts\textstyle
\let\dpl\displaystyle

\newcommand{\lcr}{[\![}
\newcommand{\rcr}{]\!]}
\newcommand{\lpr}{(\!(}
\newcommand{\rpr}{)\!)}

\newcommand\cD{\mathcal{D}}
\newcommand\cF{\mathcal{F}}
\newcommand\cH{\mathcal{H}}
\newcommand\cI{\mathcal{I}}
\newcommand\cM{\mathcal{M}}
\newcommand\cO{\mathcal{O}}
\newcommand\cT{\mathcal{T}}

\def\CC{\mathbb{C}}
\def\PP{\mathbb{P}}
\def\ZZ{\mathbb{Z}}

\newcommand{\bH}{\boldsymbol{H}}
\newcommand{\bR}{\boldsymbol{R}}

\newcommand{\bbullet}{{\scriptscriptstyle\bullet}}
\newcommand{\cbbullet}{{\raisebox{1pt}{$\bbullet$}}}
\newcommand{\isom}{\stackrel{\ts\sim}{\longrightarrow}}

\newcommand{\an}{\mathrm{an}}
\newcommand{\coh}{\mathrm{coh}}
\newcommand{\dR}{{\mathrm{dR}}}
\newcommand{\loc}{{\mathrm{loc}}}
\newcommand{\rb}{{\mathrm{b}}}
\newcommand{\rc}{{\mathrm{c}}}
\newcommand{\rd}{{\mathrm{d}}}

\DeclareMathOperator{\coker}{coker}
\DeclareMathOperator{\DR}{DR}
\DeclareMathOperator{\gr}{gr}
\DeclareMathOperator{\rk}{rk}
\DeclareMathOperator{\Spec}{Spec}
\DeclareMathOperator{\MHM}{\mathsf{MHM}}
\DeclareMathOperator{\Tot}{Tot}

\let\ra\rightarrow
\renewcommand\to{\mathchoice{\longrightarrow}{\rightarrow}{\rightarrow}{\rightarrow}}
\newcommand\mto{\mathchoice{\longmapsto}{\mapsto}{\mapsto}{\mapsto}}
\newcommand\hto{\mathrel{\lhook\joinrel\to}}

\newcommand{\To}[1]{\mathchoice{\xrightarrow{\textstyle\kern4pt#1\kern3pt}}{\stackrel{#1}{\longrightarrow}}{}{}}

\newcommand{\catD}{\mathsf{D}}
\newcommand{\Mod}{\mathsf{Mod}}

\newcommand{\bun}{{\boldsymbol{1}}}
\newcommand{\Gm}{\mathbb{G}_\mathrm{m}}
\newcommand{\Afu}{\mathbb{A}^{\!1}}
\newcommand{\Afuan}{\mathbb{A}^{\!1\an}}

\newcommand{\Afnp}{\mathbb{A}^{\!n+1}}

\newcommand{\quand}{\quad\text{and}\quad}

\begin{document}

\title{Duality for Landau-Ginzburg models}

\author{Claude Sabbah}
\institute{Claude Sabbah \at{CMLS, CNRS, École polytechnique, Institut Polytechnique de Paris, 91128 Palaiseau cedex, France}
\email{Claude.Sabbah@polytechnique.edu}}


\motto{To the memory of Bumsig Kim}

\maketitle

\abstract{This article surveys various duality statements attached to a pair consisting of a smooth complex quasi-projective variety and a regular function on it.
}

\abstract*{This article surveys various duality statements attached to a pair consisting of a smooth complex quasi-projective variety and a regular function on it.}

\section*{Contents}
\setcounter{minitocdepth}{2}
\dominitoc

\section{Introduction}

\subsection{Comparing cohomologies of complexes of differential forms}\label{subsec:comparing}
The origin of this note is a question of Bumsig Kim: given a regular function $f$ on a smooth connected complex quasi-projective variety $U$ of dimension $n$, to compare two kinds of cohomologies attached to $f$ together with their natural duality pairings:
\begin{itemize}
\item
the hypercohomology of the twisted de~Rham complex $(\Omega^\cbbullet_U,\rd+\rd f)$ together with a variant with compact support (to be defined since it is not a complex in $\Mod(\cO_U)$) and the natural pairing between them,
\item
the cohomology and cohomology with compact support of the complex $(\Omega^\cbbullet_U,\rd f)$, together with the Serre duality pairing between them (\cf\eg\cite{Hartshorne72}).
\end{itemize}

A first observation is that, although the first (hyper)cohomologies are finite dimensional, the second ones are not, unless some assumption on $f$ is added, \eg the critical set, which is the support of the cohomology sheaves $\cH^j(\Omega^\cbbullet_U,\rd f)$, is compact (recall that it is anyway contained in a finite number of fibers of $f$). This leads us to choose a good projectivization of $(U,f)$ as a projective morphism $f:Y\to\Afu$, so that $Y$ is smooth quasi-projective and $H:=Y\moins U$ is a divisor with normal crossings, and to consider suitable cohomologies on $Y$.

Furthermore, one way to compare them is to introduce a parameter $u$ and to consider the twisted de~Rham complex $(\Omega^\cbbullet_U[u],u\rd+\rd f)$, where, all along this note, the notation $[u]$ means the tensor product $\otimes_{\CC}\CC[u]$. The main question is then whether the hypercohomology $\bH^k(U,(\Omega^\cbbullet_U[u],u\rd+\rd f))$ of the latter complex is a free $\CC[u]$-module of finite rank.

\begin{enumeratei}
\item\label{enum:a}
This property does not hold in general, as shown by the following simple example. Let $U=\Afu$ with coordinate $t$ and set $f=0$. Then
\[
H^1(U,u\rd+\rd f)=\coker\Bigl[\CC[t,u]\To{u\partial_t}\CC[t,u]\Bigr]
\]
and this $\CC[u]$-module is identified with the $\CC$-vector space $\CC[t]$ on which $u$ acts by zero, hence is not of finite type.

\item\label{enum:b}
This property holds if $f$ is \emph{proper}, but we are mainly interested in cases where~$f$ is not proper.

\item\label{enum:c}
This property holds if we replace $\CC[u]$ with the ring of Laurent polynomials $\CC[u,u^{-1}]$, meaning that the problem of comparison only occurs at the origin $u=0$ (\cf \eqref{eq:pairuuum}).

\item\label{enum:e}
As indicated above, a way to prevent the infinite dimensionality at $u=0$ is to replace the twisted de~Rham complex with parameter $u$ by a subcomplex that solves the comparison. This can be done by choosing a good projectivization $(Y,f)$ as above and by considering the twisted logarithmic complex $(\Omega^\cbbullet_Y(\log H)[u],u\rd+\rd f)$ with parameter $u$ (a meromorphic version of this complex, called the de~Rham complex of the Brieskorn lattice, is considered in \cite[\S8]{S-Y14}). This will be our starting point.

The abutment of the pairing we look for is obtained by means of a smooth projectivization $X$ of $Y$ such that $f$ extends as a morphism $f:X\to\PP^1$ and $D=X\moins U$ is a normal crossing divisor. The de~Rham hypercohomologies $\bH^k(X,(\Omega_X^\cbbullet,\rd))$ and $\bH^k(X,(\Omega_X^\cbbullet,0))$ have the same dimension due to Hodge degeneration. It follows that the hypercohomology of the complex $(\Omega_X^\cbbullet[u],u\rd)$ is $\CC[u]$-free of finite rank: its has finite type over $\CC[u]$ and for any $u_o\in\CC$, the dimension of the hypercohomology of $(\Omega_X^\cbbullet,u_o\rd)$ is independent of $u_o$, so that the assertion follows from Lemma~\ref{lem:basicu} below. We~can then identify $\bH^{2n}(X,(\Omega_X^\cbbullet[u],u\rd))$ with
\[
H^{2n}_\dR(X,(\Omega_X^\cbbullet,\rd))\otimes\CC[u]\simeq H^n(X,\Omega_X^n)\otimes\CC[u]\simeq\CC[u].
\]

\begin{theoremeA}\label{th:A}
The $\CC[u]$-modules
\[
\bH^k\bigl(Y,(\Omega_Y^\cbbullet(\log H)[u],u\rd+\rd f)\bigr)\text{ and } \bH^k\bigl(Y,(\Omega_Y^\cbbullet(\log H)(-H)[u],u\rd-\rd f)\bigr),
\]
are $\CC[u]$-free of finite rank, and equipped with a meromorphic connection having a pole of order at most~$2$ at $u=0$, a regular singularity at infinity and no other pole. Furthermore, there is a natural perfect pairing
\begin{multline}\label{eq:pairY}
\bH^{n+k}\bigl(Y,(\Omega_Y^\cbbullet(\log H)[u],u\rd+\rd f)\bigr)\\\otimes_{\CC[u]} \bH^{n-k}\bigl(Y,(\Omega_Y^\cbbullet(\log H)(-H)[u],u\rd-\rd f)\bigr)
\to\CC[u],
\end{multline}
which is compatible with the connections. All these objects are independent of the choice of the good projectivization $(Y,f)$ of $(U,f)$.
\end{theoremeA}

Note that the freeness of $\bH^k\bigl(Y,(\Omega_Y^\cbbullet(\log H)[u],u\rd+\rd f)\bigr)$ also follows from a variant of the Barannikov-Kontsevich theorem \cite[\S0.6]{Bibi97b}. The independence on the choice of the good projectivization follows from \cite{Yu12} and \cite[Prop.\,2.3]{C-Y16}.

Let us emphasize that, by restricting modulo $u\CC[u]$ according to Lemma \ref{lem:basicu} below, we find that the perfect pairing
\[
\bH^{n+k}(Y,(\Omega^\cbbullet_Y(\log H),\rd f))\otimes_{\CC}\bH^{n-k}(Y,(\Omega^\cbbullet_Y(\log H)(-H),-\rd f))\to\CC
\]
coincides with Serre's duality pairing as constructed in \cite{Hartshorne72} since the complexes involved are compactly supported (they are supported on the union of the closures of the critical locus of the restriction of $f$ to each stratum of $(Y,H)$, hence in a finite number of fibers of $f$).

If $f:U\to\Afu$ is proper, the divisor $H$ is empty and $U=Y$, so this theorem justifies the assertion \eqref{enum:b}. If on the other hand $f$ satisfies a tameness condition (\cf Definition \ref{def:tame} below), then the twist $(-H)$ on the second term of \eqref{eq:pairY} can be omitted (\cf Corollary \ref{cor:tame}).

\item\label{enum:f}
One can give a $\cD$-module-theoretic interpretation of the previous results (\cf Section \ref{sec:RF}). Letting $\cD_U$ denote the ring of algebraic differential operators on $U$ with its filtration~$F_\bbullet$ by the order, we consider the Rees ring $R_F\cD_U=\bigoplus_{k\geq0}F_k\cD_U\cdot u^k$. We~regard $(\cO_U[u],u\rd+\rd f)$ as a coherent $R_F\cD_U$-module and $\bH^k(U,(\Omega^\cbbullet_U[u],u\rd+\rd f))$ as isomorphic to the de~Rham cohomology of this $R_F\cD_U$-module. Although the ring $R_F\cD_U$ has properties similar to those of $\cD_U$, it does not satisfy Bernstein's inequality because of the possible occurrence of $u$-torsion, and this prevents us to apply Bernstein's results to deduce finiteness of the de~Rham cohomology.
\end{enumeratei}

For \eqref{enum:c}, we first notice that the second term in \eqref{eq:pairY} plays the role of hypercohomology with compact support. In order to give a meaning to this remark, it~is worthwhile working with modules over the ring of differential operators. More precisely, we consider the ring $\cD_U[u,u^{-1}]=\cD_U\otimes_{\CC}\CC[u,u^{-1}]$ of algebraic differential operator on $U$ with coefficients in $\cO_U[u,u^{-1}]$, so that the base ring is $\CC[u,u^{-1}]$ instead of the field $\CC$ (it would be equivalent to consider differential operators on $U\times\Gm$ relative to the projection to $\Gm$). We~denote by $E_U^{f/u}$ the left $\cD_U[u,u^{-1}]$-module $(\cO_U[u,u^{-1}],\rd+\rd f/u)$.

We consider the two extensions in the sense of $\cD$-modules, denoted $E_Y^{f/u}(*H)$ and $E_Y^{f/u}(!H)$ of $E_U^{f/u}$ by the open inclusion $U\hto Y$, corresponding to the full extension and the extension with ``proper support'' (\cf Section \ref{sec:pfB} for details). Note that $E_Y^{f/u}$ also exists. We~will show that, for any $k\in\ZZ$, both $\CC[u,u^{-1}]$-modules $H^k_\dR(Y,E_Y^{f/u}(*H))$ and $H^k_\dR(Y,E_Y^{f/u}(!H))$ are free $\CC[u,u^{-1}]$-modules of finite rank equipped with a connection having a regular singularity at $u=\infty$, that we also denote by $H^k_\dR(U,E_U^{f/u})$ and $H^k_{\dR,\rc}(U,E_U^{f/u})$ respectively. By definition of $E_Y^{f/u}(*H)$, we have
\[
H^k_\dR(U,E_U^{f/u})=\bH^k(U,(\Omega^\cbbullet_U[u,u^{-1}],\rd+\rd f/u)).
\]
On the other hand, we will see the isomorphism in \eqref{eq:inductive**}:
\begin{multline*}
\bH^k\bigl(Y,(\Omega_Y^\cbbullet(\log H)[u,u^{-1}],\rd+\rd f/u)\\\simeq\bH^k\bigl(Y,(\Omega_Y^\cbbullet(\log H)[u],\rd+\rd f/u)\otimes_{\CC[u]}\CC[u,u^{-1}],
\end{multline*}
compatible with the connections, and similarly after twisting the complexes by $(-H)$.

\begin{theoremeA}\label{th:B}
Restriction to $U$ induces isomorphisms of free $\CC[u,u^{-1}]$-modules of finite rank with connection
\begin{align*}
\bH^{n+k}\bigl(Y,(\Omega_Y^\cbbullet(\log H)[u,u^{-1}],\rd+\rd f/u)\bigr)&\isom H^k_\dR(U,E_U^{f/u})\tag{B$*$}\label{eq:B*}\\
\bH^{n-k}\bigl(Y,(\Omega_Y^\cbbullet(\log H)(-H)[u,u^{-1}],\rd-\rd f/u)\bigr)&\isom H^{-k}_{\dR,\rc}(U,E_U^{-f/u}),\tag{B$!$}\label{eq:B!}
\end{align*}
giving rise to a perfect pairing compatible with the connections, by means of \eqref{eq:pairY}:
\begin{equation}\label{eq:pairuuum}
H^k_\dR(U,E_U^{f/u})\otimes_{\CC[u,u^{-1}]}H^{-k}_{\dR,\rc}(U,E_U^{-f/u})\to \CC[u,u^{-1}].
\end{equation}
\end{theoremeA}

This theorem makes clear the independence of the choice of the good projectivization $(Y,f)$ in Theorem \ref{th:A} if we restrict to $\Gm=\Spec\CC[u,u^{-1}]$. It is similar to \cite[Th.\,3]{Wei17}, where mixed Hodge modules are considered (\cf also \cite{E-S18}). Note that a perfect pairing like that of Theorem~\ref{th:B} can be obtained by means of $\cD$-module theory (\cf\cite[App.\,2]{Malgrange91}) but we will not compare these two ways of defining a perfect pairing. Theorem \ref{th:A} provides each term in \eqref{eq:B*} and \eqref{eq:B!} of a canonical lattice, \ie a free $\CC[u]$-submodule which has the same rank as the corresponding $\CC[u,u^{-1}]$-module, on which the connection has a pole of order two. It is a generalization of the \emph{Brieskorn lattice} of singularity theory (\cite{Brieskorn70,Pham85b}).

\begin{remark}
Such twisted de~Rham complexes have been considered from many points of view and their cohomology sometimes takes the name of Dwork cohomology. Also, since $\rd f$ is the main object used to define the twisted de~Rham complex, one could consider a pair $(U,\omega)$, where $\omega$ is any closed regular $1$-form, instead of a pair $(U,f)$.  In addition to the references used in this text, let us mention only a few articles representing other directions of research:
\cite{A-S97}, \cite{D-M-S-S99}, \cite{B-D04}, \cite{Hertling05}, \cite{S-T08}, \cite{Fan11}, \cite{L-W22} and \cite{Arapura97}, \cite[\S2.4]{Bibi13b}.
\end{remark}

\subsection{Formalization with respect to \texorpdfstring{$u$}{u}}
If the critical set of $f$ in $U$ is assumed to be compact, then the cohomology $\bH^k(U,(\Omega_U^\cbbullet,\rd f))$ is finite-dimensional since the complex of coherent sheaves is supported on this critical set. One would naturally expect that, consequently, $\bH^k(U,(\Omega^\cbbullet_U[u],u\rd+\rd f))$ is of finite type over $\CC[u]$. We~do not know whether this property holds. However, as shown in Theorem \ref{th:Ch} below, it~holds if we replace $\CC[u]$ with the ring $\CC\lcr u\rcr$ of formal power series in a suitable way, under the following condition:\footnote{Added on May 2023: T.\,Mochizuki has sent me a proof assuming only compactness of the critical set and which extends to the case where $U$ is complex analytic and Kähler, and $f$ holomorphic.}

\smallskip\noindent\refstepcounter{equation}\label{cond:b}
\eqref{cond:b} 
There exists a projectivization $f_Z:Z\to\Afu$ of $f:U\to\Afu$ with $Z$ smooth (no other assumption on $Z\moins U$), such that the critical set of $f_Z$ is contained in $U$ (in~particular, it is compact).

\smallskip
Some care has to be taken when working with formal power series. For a coherent $\cO_U$-module $\cF$, we have $\cF[u]:=\CC[u]\otimes_{\CC}\cF$ and we set $\cF\lcr u\rcr=\varprojlim_\ell(\cF[u]/u^\ell\cF[u])$. This is in general not equal to $\CC\lcr u\rcr\otimes_{\CC}\cF$ and for $x\in U$ we have a strict inclusion $\cF\lcr u\rcr_x\subsetneq \cF_x\lcr u\rcr$: a germ of section of $\cF\lcr u\rcr$ at $x\in U$ consists of a formal power series $\sum_nf_nu^n$ where $f_n$ are sections of $\cF$ defined on a fixed neighbourhood of $x$, while for $\cF_x\lcr u\rcr$ we allow the neighbourhood to be shrunk when $n\to\infty$. In particular, there is a natural morphism $\CC\lcr u\rcr\otimes_{\CC}\cF\to\cF\lcr u\rcr$. We~then set $\cF\lpr u\rpr:=\cF\lcr u\rcr[u^{-1}]=\CC\lpr u\rpr\otimes_{\CC\lcr u\rcr}\cF\lcr u\rcr$.

\begin{theoremeAh}\label{th:Ah}
The statement of Theorem \ref{th:A} holds if we replace everywhere $[u]$ with~$\lcr u\rcr$, and the free $\CC\lcr u\rcr$-modules and pairings are obtained from those of Theorem \ref{th:A} by tensoring with $\CC\lcr u\rcr$ over $\CC[u]$.
\end{theoremeAh}

\refstepcounter{theoremeAh}
\begin{theoremeAh}\label{th:Ch}
Under Condition \eqref{cond:b}, the $\CC\lcr u\rcr$-modules
\[
\bH^k\bigl(U,(\Omega_U^\cbbullet\lcr u\rcr,u\rd+\rd f)\bigr)
\]
are $\CC\lcr u\rcr$\nobreakdash-free, coincide via the restriction morphism with the corresponding~$\CC\lcr u\rcr$\nobreakdash-mod\-ules in Theorem \ref{th:Ah}, and the pairing \eqref{eq:pairY} induces a perfect pairing
\begin{equation}\label{eq:pairC}
\bH^{n+k}\bigl(U,(\Omega_U^\cbbullet\lcr u\rcr,u\rd+\rd f)\bigr)\otimes_{\CC[u]} \bH^{n-k}\bigl(U,(\Omega_U^\cbbullet\lcr u\rcr,u\rd-\rd f)\bigr)
\to \CC\lcr u\rcr
\end{equation}
which itself induces, by working modulo $u\CC\lcr u\rcr$, the Serre duality pairing
\[
\bH^{n+k}\bigl(U,(\Omega_U^\cbbullet,\rd f)\bigr)\otimes_{\CC} \bH^{n-k}\bigl(U,(\Omega_U^\cbbullet,-\rd f)\bigr)
\to\CC.
\]
\end{theoremeAh}

\Cf Proposition \ref{prop:C} for a more precise result. The way \eqref{eq:pairY} induces \eqref{eq:pairC} will be explained in detail in Section \ref{sec:pfC}.

\begin{remark}
The pairing \eqref{eq:pairC} can be regarded as a global version (with respect to~$U$) of K.\,Saito's higher residue pairings \cite{KSaito83} for a germ of holomorphic function with an isolated singularity.
\end{remark}

\subsection{Setting and notation}\label{subsec:nota}
Let $U$ be a smooth connected quasi-projective variety of dimension $n$ and let $f\in\cO(U)$ be any regular function on $U$, that we regard as a morphism $f:U\to\Afu$. Let $u$ be a new variable. We~consider the \emph{twisted de~Rham complex} $(\Omega_U^\cbbullet[u],u\rd+\rd f)$, with $\Omega_U^\cbbullet[u]:=\Omega_U^\cbbullet\otimes_{\CC}\CC[u]$, whose hypercohomology on $U$ is \hbox{$\bH^k(U,(\Omega^\cbbullet_U[u],u\rd+\rd f))$}. We~sometimes make use of the isomorphic subcomplex $(u^{-\cbbullet}\Omega_U^\cbbullet[u],\rd+\rd f/u)$ of $(\Omega_U^\cbbullet[u,u^{-1}],\rd+\rd f/u)$, the isomorphism being obtained by multiplying the degree $k$ term by $u^{-k}$. We~will also define the cohomology with compact support $\bH^k_\rc(U,(\Omega^\cbbullet_U[u],u\rd+\rd f))\simeq\bH^k_\rc(U,(u^{-\cbbullet}\Omega^\cbbullet_U(-D)[u],\rd+\rd f/u))$.

It is convenient to choose a \emph{good projectivization} of $(U,f)$, namely, a pair $(X,f)$ consisting smooth projective variety $X$ containing $U$ as a Zariski open subset and such that
\begin{enumeratea}
\item
$D:=X\moins U$ is a normal crossing divisor in $X$,
\item
$f:U\to\Afu$ extends as a morphism $f:X\to\PP^1$.
\end{enumeratea}
We then set $P=f^*(\infty)$ (the pole divisor of $f$), we denote by $|P|$ its support, which is contained in $D$, and we decompose $D=|P|\cup H$, where $H$ is some normal crossing divisor in $X$ having no irreducible component contained in $|P|$. We~let $j:U:=X\moins D\hto X$ denote the open inclusion. We~note that $f$ induces a projective morphism $f:X\moins|P|=:Y\to\Afu$. We~can also regard $f$ as a global section of~$\cO_X(*D)$.  We will keep the notation~$H$ for $H\cap Y$.

The following lemma will be of constant use.

\begin{lemma}\label{lem:basicu}
Let $K^\cbbullet_u$ be a bounded complex of sheaves on $X$ of $\CC[u]$-modules such that $\bH^j(X,K^\cbbullet_u)$ has finite type over $\CC[u]$ for every $j$. Then the following properties are equivalent:
\begin{enumerate}
\item\label{lem:basicu1}
For every $u_o\in\CC$ and every $j$, $\dim \bH^j(X,K^\cbbullet_u/(u-u_o)K^\cbbullet_u)$ is independent of~$u_o$.
\item\label{lem:basicu2}
For every $j$, $\bH^j(X,K^\cbbullet_u)$ is a free $\CC[u]$-module.\end{enumerate}
In such a case, for every $u_o\in\CC$ and every $j$, we have
\begin{equation}\tag{\ref{lem:basicu}$\,*$}\label{eq:basicu}
\bH^j(X,K^\cbbullet_u/(u-u_o)K^\cbbullet_u)=\bH^j(X,K^\cbbullet_u)/(u-u_o)\bH^j(X,K^\cbbullet_u).
\end{equation}
Furthermore, for any morphism $\varphi:K^\cbbullet_u\to L^\cbbullet_u$ between two such complexes satisfying~\eqref{lem:basicu1} or \eqref{lem:basicu2}, if the induced morphism
\[
\bH^j(X,K^\cbbullet_u/(u-u_o)K^\cbbullet_u)\to \bH^j(X,L^\cbbullet_u/(u-u_o)L^\cbbullet_u)
\]
is an isomorphism for any $u_o\in\CC$ and any $j\in\ZZ$, then $\bR\Gamma(X,\varphi)$ is a quasi-isomorphism, that is,
\[
\bH^j(X,\varphi):\bH^j(X,K^\cbbullet_u)\to\bH^j(X,L^\cbbullet_u)
\]
is an isomorphism of free $\CC[u]$-modules for any $j\in\ZZ$.
\end{lemma}

\begin{proof}
Let $\CC[u]_\loc$ be a localization of $\CC[u]$ such that $\CC[u]_\loc\otimes_{\CC[u]}\bH^j(X,K^\cbbullet_u)$ is $\CC[u]_\loc$-free for every $j$. If $u_o$ is a closed point of $\Spec\CC[u]_\loc$, the maps $(u-u_o)$ in the long exact sequence
\[
\cdots \bH^j(X,K^\cbbullet_u{})\To{u-u_o}\bH^j(X,K^\cbbullet_u{})\to \bH^j(X,K^\cbbullet_u/(u-u_o)K^\cbbullet_u)\to\cdots
\]
are all injective. By decreasing induction on $j$, one identifies $\bH^j(X,K^\cbbullet_u/(u-u_o)K^\cbbullet_u)$ with $\bH^j(X,K^\cbbullet_u)/(u-u_o)\bH^j(X,K^\cbbullet_u)$ for all $j$, \ie that \eqref{eq:basicu} holds for such an~$u_o$.

If \eqref{lem:basicu2} holds, then the above property holds for any $u_o\in\CC$ and the dimension of $\bH^j(X,K^\cbbullet_u/(u-u_o)K^\cbbullet_u)$ is constant and equal to the rank of $\bH^j(X,K^\cbbullet_u)$, so that \eqref{lem:basicu1} also holds.

Assume now that \eqref{lem:basicu1} holds. We~argue by decreasing induction on~$j$. Assume that $u-u_o:\bH^{j+1}(X,K^\cbbullet_u)\to\bH^{j+1}(X,K^\cbbullet_u)$ is injective for any $u_o$. Then the exact sequence above implies that  \eqref{eq:basicu} holds in degree $j$ for any $u_o$. Since the dimension is independent of $u_o$, the $\CC[u]$-module $\bH^j(X,K^\cbbullet_u)$ is $\CC[u]$-free, so that $u-u_o$ is injective on it for any $u_o\in\CC$ and we conclude by induction since $\bH^k(X,K^\cbbullet_u)=0$ for $k\gg0$.

For the last assertion, the assumption and the first part imply that $\varphi$ induces an isomorphism $\bH^j(X,K^\cbbullet_u)/(u-u_o)\bH^j(X,K^\cbbullet_u)\isom\bH^j(X,L^\cbbullet_u)/(u-u_o)\bH^j(X,L^\cbbullet_u)$ for any $u_o\in\CC$. We~conclude by applying a variant of Nakayama's lemma: if a morphism between free $\CC[u]$-modules of finite rank induces an isomorphism after restriction to any $u_o\in\CC$, then it is an isomorphism.\qed
\end{proof}

\section{Freeness and duality for the Kontsevich complexes}\label{sec:Kontsevich}
Before considering the hypercohomology $\bH^k(U,(\Omega^\cbbullet_U[u],u\rd+\rd f))$, it is useful to gather some properties of a variant of this de~Rham cohomology where the computation is made on the projective variety $X$ and the terms of the de~Rham complexes are $\cO_X$-coherent. If $f=0$, this amounts to computing the hypercohomology of the logarithmic de~Rham complex instead of that of the meromorphic de~Rham complex on $X$ (\cf Notation of Section \ref{subsec:nota}).

\subsection{Kontsevich complexes \textup{\cite{K-K-P14}}}

For $k\geq0$, we set
\[
\Omega^k_f=\{\omega\in\Omega_X^k(\log D)\mid\rd f\wedge\omega\in\Omega_X^{k+1}(\log D)\}.
\]
Since $\rd$ sends $\Omega_X^k(\log D)$ to $\Omega_X^{k+1}(\log D)$, we obtain the \emph{Kontsevich complex}
\[
(\Omega^\cbbullet_f,\rd+\rd f)\qquad\text{(denoted $\Omega^\cbbullet_X(\log D,f)$ in \cite{K-K-P14})}.
\]
We will also consider the twisted complex $(\Omega^\cbbullet_f(-|P]),\rd+\rd f)$, whose terms are $\Omega^k_f(-|P|)$.

One can equip these complexes with the decreasing filtration $\sigma^{\cbbullet}$ by the stupidly truncated subcomplexes. The inclusion of filtered complexes
\[
(\Omega^\cbbullet_f(-|P|),\rd+\rd f,\sigma^{\cbbullet})\hto(\Omega^\cbbullet_f,\rd+\rd f,\sigma^{\cbbullet})
\]
and
\[
(\Omega^\cbbullet_f(-D),\rd+\rd f,\sigma^{\cbbullet})\hto(\Omega^\cbbullet_f(-H),\rd+\rd f,\sigma^{\cbbullet})
\]
are filtered quasi-isomorphisms (\cf\cite[Prop.\,1.4.2]{E-S-Y13} and \cite[Proof of Lem.\,2.12]{K-K-P14}).

The local computation of $\Omega^k_f$ (\cf\cite[(1.3.1)]{E-S-Y13}) shows that the wedge product
\begin{equation}\label{eq:dualomegaf}
\Omega^k_f\otimes_{\cO_X}\Omega^{n-k}_f(-D)\to\Omega^n_X(\log D)(-D)=\Omega^n_X
\end{equation}
is a perfect pairing. It also induces a pairing of complexes
\begin{equation}\label{eq:pairing}
(\Omega^\cbbullet_f,\rd+\rd f)\otimes_{\CC}(\Omega^\bbullet_f(-D),\rd-\rd f)\to(\Omega^\cbbullet_X,\rd),
\end{equation}
where the termwise product is induced by
\[
\Omega_X^k(\log D)\otimes\Omega_X^\ell(\log D)(-D)\to\Omega_X^{k+\ell}(\log D)(-D)\hto\Omega_X^{k+\ell}.
\]

\begin{proposition}[J.-D.\,Yu \cite{Yu12}]\label{prop:Kontsevichperfect}
The corresponding cohomological pairing
\[
\bH^{n+k}\bigl(X,(\Omega^\cbbullet_f,\rd+\rd f)\bigr)\otimes_{\CC}\bH^{n-k}\bigl(X,(\Omega^\cbbullet_f(-D),\rd-\rd f)\bigr)\to H^{2n}_{\dR}(X)
\]
is perfect.
\end{proposition}

\begin{corollary}
Through the quasi-isomorphism \hbox{$(\Omega^\cbbullet_f(-|P|),\rd\!+\!\rd f)\!\hto\!(\Omega^\cbbullet_f,\rd\!+\!\rd f)$}, the pairing obtained from that of Proposition \ref{prop:Kontsevichperfect}:
\[
\bH^{n+k}\bigl(X,(\Omega^\cbbullet_f,\rd+\rd f)\bigr)\otimes_{\CC}\bH^{n-k}\bigl(X,(\Omega^\cbbullet_f(-H),\rd-\rd f)\bigr)\to H^{2n}_{\dR}(X)
\]
is perfect.\qed
\end{corollary}

\begin{proof}[of Proposition \ref{prop:Kontsevichperfect}]\qed
We will make use of the following lemma.

\begin{lemma}[J.-D.\,Yu]\label{lem:Yu}
Let $A^\cbbullet,B^\cbbullet$ be bounded complexes of $\cO_X$-modules equipped with
\begin{itemize}
\item
finite exhaustive decreasing filtrations $F^\cbbullet$,
\item
a pairing $A^\cbbullet\otimes B^\cbbullet\to(\Omega^\cbbullet_X,\rd)$
\end{itemize}
satisfying the two conditions
\begin{enumerate}
\item\label{lem:Yu1}
the pairing induces a well-defined pairing $F^pA^\cbbullet\otimes (B^\cbbullet/F^{n+1-p}B^\cbbullet)\to(\Omega^\cbbullet_X,\rd)$ for each $p$,
\item\label{lem:Yu2}
the induced pairing $\bH^{n+k}\bigl(X,\gr^p_FA^\cbbullet\bigr)\otimes_{\CC}\bH^{n-k}\bigl(X,\gr^{n-p}_FB^\cbbullet\bigr)\to H^{2n}_{\dR}(X)
$ is perfect for each $k,p$.
\end{enumerate}
Then the induced pairing $\bH^{n+k}\bigl(X,A^\cbbullet\bigr)\otimes_{\CC}\bH^{n-k}\bigl(X,B^\cbbullet\bigr)\to H^{2n}_{\dR}(X)
$ is perfect for each $k$.
\end{lemma}

\begin{proof}
We argue by induction on $p$. We~consider the commutative diagram (omitting~$X$ in the notation and setting $F^p:=F^pA^\cbbullet$ and $G_{n+1-p}:=B^\cbbullet/F^{n+1-p}B^\cbbullet$)
\begin{small}
\[
\xymatrix@C.2cm{
\bH^{n+k-1}(\gr^p_F)\ar[d]\ar[r]
&\bH^{n+k}(F^{p+1})\ar[d]\ar[r]
&\bH^{n+k}(F^p)\ar[d]\ar[r]
&\bH^{n+k}(\gr^p_F)\ar[d]\ar[r]
&\bH^{n+k+1}(F^{p+1})\ar[d]
\\
\bH^{n-k+1}(\gr^{n-p}_F)^\vee\ar[r]
&\bH^{n-k}(G_{n-p})^\vee\ar[r]
&\bH^{n-k}(G_{n+1-p})^\vee\ar[r]
&\bH^{n-k}(\gr^{n-p}_F)^\vee\ar[r]
&\bH^{n-k-1}(G_{n-p})^\vee
}
\]
\end{small}%
By decreasing induction on $p$ and Condition \eqref{lem:Yu2}, the vertical morphisms except maybe the middle one are isomorphisms. Hence the middle one is so. For $p\gg0$, both terms $\bH^{n+k}(F^p)$ and $\bH^{n-k}(G_{n+1-p})^\vee$ are zero, and for $p\ll0$,
\[
\bH^{n+k}(F^p)=\bH^{n+k}(X,A^\cbbullet)\quand\bH^{n-k}(G_{n+1-p})^\vee=\bH^{n-k}(X,B^\cbbullet).\eqno\qed
\]
\end{proof}

If we equip each complex in \eqref{eq:pairing} with the filtration $\sigma^{\cbbullet}$, the pairing \eqref{eq:pairing} clearly satisfies \ref{lem:Yu}\eqref{lem:Yu1}. Furthermore, \ref{lem:Yu}\eqref{lem:Yu2} follows from Serre's duality applied to \eqref{eq:dualomegaf}. Therefore, an application of Lemma \ref{lem:Yu} concludes the proof.\qed
\end{proof}

\begin{remark}\label{rem:Kontsevichperfect}\mbox{}
\begin{enumerate}
\item
A similar argument with the complex $(\Omega_X^\cbbullet,0)$ instead of $(\Omega_X^\cbbullet,\rd)$ yields a perfect pairing induced by Serre's duality:
\[
\bH^{n+k}\bigl(X,(\Omega^\cbbullet_f,\rd f)\bigr)\otimes_{\CC}\bH^{n-k}\bigl(X,(\Omega^\cbbullet_f(-D),-\rd f)\bigr)\to \bH^n(X,\Omega^n_X).
\]
As for the Kontsevich complex, the inclusion $(\Omega^\cbbullet_f(-|P|),\rd f)\hto(\Omega^\cbbullet_f,-\rd f)$ is a quasi-isomorphism, as well as $(\Omega^\cbbullet_f(-D),\rd f)\hto(\Omega^\cbbullet_f(-H),-\rd f)$, so that we deduce a perfect pairing
\[
\bH^{n+k}\bigl(X,(\Omega^\cbbullet_f,\rd f)\bigr)\otimes_{\CC}\bH^{n-k}\bigl(X,(\Omega^\cbbullet_f(-H),-\rd f)\bigr)\to \bH^n(X,\Omega^n_X).
\]
\item
There exist natural perfect pairings
\[
\bH^{n+k}\bigl(Y,(\Omega^\cbbullet(\log H),\rd+\rd f)\bigr)\otimes_{\CC}\bH^{n-k}\bigl(Y,(\Omega^\cbbullet(\log H)(-H),\rd-\rd f)\bigr)\to \CC
\]
and
\[
\bH^{n+k}\bigl(Y,(\Omega^\cbbullet(\log H),\rd f)\bigr)\otimes_{\CC}\bH^{n-k}\bigl(Y,(\Omega^\cbbullet(\log H)(-H),-\rd f)\bigr)\to\CC.
\]
This will be shown with a parameter in Lemma \ref{lem:XY} below, by identifying the source of these pairings respectively with the sources of the pairing of Proposition \ref{prop:Kontsevichperfect} and that of the previous remark.
\end{enumerate}
\end{remark}

\subsection{Kontsevich complexes with the variable \texorpdfstring{$u$}{u}}\label{subsec:Kontsevich}
We now replace $(\Omega_f^\cbbullet,\rd+\rd f)$ with $(\Omega_f^\cbbullet[u],u\rd+\rd f)$. If we make $u^2\partial_u$ act on $\Omega_f^k[u]$ by
\[
u^2\partial_u(\eta\otimes h(u))=\eta\otimes(u^2\partial_u+ku)(h(u))-f\eta\otimes h(u),
\]
then this action commutes with the differential of the complex and induces a natural action on its hypercohomology, \ie a meromorphic connection with a pole of order two at $u=0$ and no other pole except at infinity.

\begin{corollary}\label{cor:dualiteuddf}
The cohomologies
\[
\bH^{n+k}\bigl(X,(\Omega_f^\cbbullet[u],u\rd+\rd f)\bigr)\quad\text{and}\quad \bH^{n-k}\bigl(X,(\Omega_f^\cbbullet(-D)[u],u\rd+\rd f)\bigr)
\]
are $\CC[u]$-free of finite rank, and the pairing
\[
\bH^{n+k}\bigl(X,(\Omega_f^\cbbullet[u],u\rd+\rd f)\bigr)\otimes_{\CC[u]}\bH^{n-k}\bigl(X,(\Omega_f^\cbbullet(-D)[u],u\rd+\rd f)\bigr)
\to H^{2n}_{\dR}(X)[u]
\]
is perfect and compatible with the natural meromorphic action of $\partial_u$. Moreover, the natural morphism
\[
\bH^{n-k}\bigl(X,(\Omega_f^\cbbullet(-D)[u],u\rd+\rd f)\bigr)\to\bH^{n-k}\bigl(X,(\Omega_f^\cbbullet(-H)[u],u\rd+\rd f)\bigr)
\]
is an isomorphism which induces, by means of the previous pairing, a perfect pairing
\[
\bH^{n+k}\bigl(X,(\Omega_f^\cbbullet[u],u\rd+\rd f)\bigr)\otimes_{\CC[u]}\bH^{n-k}\bigl(X,(\Omega_f^\cbbullet(-H)[u],u\rd+\rd f)\bigr)\\
\to H^{2n}_{\dR}(X)[u].
\]
\end{corollary}

\begin{proof}[Proof of Corollary \ref{cor:dualiteuddf}]
Recall that \hbox{$\dim\bH^k\bigl(X,(\Omega^\cbbullet_f,u_o\rd+\rd f)\bigr)$} is independent of $u_o$ (\cf\cite[Th.\,1.3.2]{E-S-Y13}). Since the $\CC[u]$-finiteness is clear by a spectral sequence argument owing to the fact that each term of the complex is $\cO_X[u]$-coherent, the $\CC[u]$-freeness of $\bH^k\bigl(X,(\Omega^\cbbullet_f[u],u\rd+\rd f)\bigr)$ follows from Lemma \ref{lem:basicu}. By duality (Proposition~\ref{prop:Kontsevichperfect} and Remark \ref{rem:Kontsevichperfect}), $\dim\bH^k\bigl(X,(\Omega^\cbbullet_f(-D),u_o\rd-\rd f)\bigr)$ is independent of $u_o$, and since $\CC[u]$-finiteness is also clear, $\CC[u]$-freeness follows.

We deduce the perfectness of the pairing by tensoring with $\CC[u]/(u-u_o)$ for any $u_o$, where it follows from \loccit\ The last assertion follows then from Remark \ref{rem:Kontsevichperfect}.\qed
\end{proof}

\begin{remark}
From the point of view developed in Theorem \ref{th:B} and the other proof given in Section \ref{sec:RF}, it is convenient to consider the complexes  $(u^{-\cbbullet}\Omega_f^\cbbullet[u],\rd\pm\rd f/u)$ with degree $k$ term $u^{-k}\Omega_f^k[u]\subset\Omega_f^k[u,u^{-1}]$. Multiplication by $u^k$ on the degree~$k$ term induces an isomorphism with $(\Omega_f^\cbbullet[u],u\rd\pm\rd f)$. We~deduce isomorphisms
\[
\bH^k\bigl(X,(u^{-\cbbullet}\Omega_f^\cbbullet[u],\rd\pm\rd f/u)\bigr)\simeq\bH^k\bigl(X,(\Omega^\cbbullet_f[u],u\rd\pm\rd f)\bigr).
\]
Due to the perfect pairing $u^{-j}\Omega_f^j[u]\otimes u^{j-n}\Omega_f^{n-j}(-D)[u]\!\to\! u^{-n}\Omega_X^n[u]$ obtained from~\eqref{eq:dualomegaf}, we see that the perfect pairing between these free $\CC[u]$-modules takes values in $u^{-n}H^{2n}_{\dR}(X)[u]$.
\end{remark}

\subsection{Proof of Theorem \ref{th:A}}
A first part of the theorem, namely $\CC[u]$-freeness and finiteness, as well as perfectness of \eqref{eq:pairY}, follows from Corollary \ref{cor:dualiteuddf}, according to the next lemma.

\begin{lemma}\label{lem:XY}
For each $k$, the natural morphisms
\begin{multline*}
\bH^k\bigl(X,(\Omega_f^\cbbullet[u],u\rd+\rd f)\to\bH^k\bigl(X,(\Omega_f^\cbbullet(*P)[u],u\rd+\rd f)\bigr)\\
\to\bH^k\bigl(Y,(\Omega^\cbbullet_Y(\log H)[u],u\rd+\rd f)\bigr),
\end{multline*}
and the similar ones after twisting the complexes by $(-H)$, are isomorphisms.
\end{lemma}

\begin{proof}
We will treat the case without twist, the other case being treated similarly. The proof of \cite[Cor.\,1.4.3]{E-S-Y13} shows that, for any $u_o\in\CC$ and any $\ell\geq1$, the inclusion of complexes
\[
(\Omega^\cbbullet_f,u_o\rd+\rd f)\hto(\Omega^\cbbullet_f(\ell P),u_o\rd+\rd f)
\]
is a quasi-isomorphism. Since $\bH^k\bigl(X,(\Omega^\cbbullet_f(\ell P)[u],u\rd+\rd f)\bigr)$ has finite type over $\CC[u]$, we can apply Lemma \ref{lem:basicu} to deduce that the morphism
\[
\bH^k\bigl(X,(\Omega_f^\cbbullet[u],u\rd+\rd f)\bigr)\to\bH^k\bigl(X,(\Omega^\cbbullet_f(\ell P)[u],u\rd+\rd f)\bigr)
\]
is an isomorphism for any $k$ and $\ell$. Since $\Omega^k_f(*P)=\varinjlim_\ell\Omega^k_f(\ell P)$, we only need to justify the commutation of direct limits and hypercohomology.

\begin{lemma}\label{lem:commutinductive}
Let $(K^\cbbullet_\ell,\delta)_\ell$ be an inductive system of complexes of fixed amplitude on an algebraic variety $Z$, whose terms are quasi-coherent $\cO_X$-modules. Then, for each~$k$, we have
\[
\varinjlim_\ell\bH^k(Z,(K^\cbbullet_\ell,\delta))\isom\bH^k(Z,\varinjlim_\ell(K^\cbbullet_\ell,\delta)).
\]
\end{lemma}

\begin{proof}
We filter the complexes by stupid truncation. We~have such a morphism at each level of the corresponding spectral sequence and it is enough to prove the assertion for the first page of the spectral sequence. This amounts to show the isomorphism
\[
\varinjlim_\ell H^k(Z,K^j_\ell)\isom H^k(Z,\varinjlim_\ell K^j_\ell),
\]
which follows from Noetherianity of $Z$, since $K^j_\ell$ are quasi-coherent.\qed
\end{proof}

To show that the second morphism is an isomorphism, we argue with the same spectral sequence argument. We~are thus reduced to showing that the restriction morphism
\[
H^k(X,\Omega^j_f(*P)[u])\to H^k(Y,\Omega^j_Y(\log H)[u])
\]
is an isomorphism. By the commutation with inductive limits, we are left with showing $H^k(X,\Omega^j_f(*P))\isom H^k(Y,\Omega^j_Y(\log H))$, which is clear since $\Omega^k_Y(\log H)$ is the restriction to $Y$ of $\Omega^k_f(*P)$.\qed
\end{proof}

\begin{proof}[of Theorem \ref{th:A}, end]
The existence of a compatible action of $u^2\partial_u$ has been seen in Section \ref{subsec:Kontsevich}. The regularity of the connection at $u=\infty$ will be seen in Remark \ref{rem:Fourier}.

It remains to show the independence of the good projectivization. It is enough to consider a morphism $\pi:(Y',H')\to(Y,H)$ of such kind which is the identity on $U$, and set $f'=f\circ\pi$. We~have natural morphisms of complexes compatible with the actions of $u^2\partial_u$ and $u^2\partial_u$ respectively:
\begin{multline*}
(\Omega^\cbbullet_Y(\log H)[u],u\rd+\rd f)\to\bR\pi_*\pi^{-1}(\Omega^\cbbullet_Y(\log H)[u],u\rd+\rd f)\\
\hspace*{4.5cm}\to\bR\pi_*(\pi^*\Omega^\cbbullet_Y(\log H)[u],u\rd+\rd f)\\
\to\bR\pi_*(\Omega^\cbbullet_{Y'}(\log H')[u],u\rd+\rd f')
\end{multline*}
inducing a natural morphism of free $\CC[u]$-modules of finite rank with meromorphic connection:
\[
\bH^k(Y,(\Omega^\cbbullet_Y(\log H)[u],u\rd+\rd f))\to\bH^k(Y',(\Omega^\cbbullet_{Y'}(\log H')[u],u\rd+\rd f')).
\]
and a similar property after twisting by $(-H)$ and $(-H')$ respectively. The perfect pairings are also compatible with these morphisms. We~are thus reduced to showing that such a morphism and its analogue after the twist is an isomorphism.

Its restriction to any $u_o\neq0$ is an isomorphism, according to \cite[Cor.\,1.4.3]{E-S-Y13}. The argument for the twisted case will be given in the proof of Theorem \ref{th:B}. To~conclude with Lemma \ref{lem:basicu} it remains to show that the natural morphism
\[
\bH^k(Y,(\Omega^\cbbullet_Y(\log H),\rd f))\to\bH^k(Y',(\Omega^\cbbullet_{Y'}(\log H'),\rd f'))
\]
is an isomorphism, and similarly after a twist by $(-H)$ and $(-H')$ respectively. The non twisted case is proved in \cite[Prop.\,2.3]{C-Y16}. The twisted case can be obtain by duality, by showing that the perfect pairings considered in Remark \ref{rem:Kontsevichperfect} are compatible with the isomorphisms induced by $\pi$. This ends the proof of Theorem \ref{th:A}.\qed
\end{proof}

\begin{remark}[Computation of the rank]\label{rem:computrk}
For $f:U\to\Afu$ as above, let $j_Y:U\hto Y$ denote the inclusion. We~consider the complex $\bR j_{Y*}\CC_U$ on $Y$, and for any $c\in\CC$, the vanishing cycle complex $\phi_{f-c}\bR j_{Y*}\CC_U$. It follows from \cite[Th.\,2]{Bibi97b} that
\[
\rk\bH^k\bigl(Y,(\Omega^\cbbullet_Y(\log H)[u],u\rd+\rd f)\bigr)=\sum_{c\in\CC}\dim\bH^{k-1}(f^{-1}(c),\phi_{f-c}\bR j_{Y*}\CC_U).
\]
\end{remark}

\section{The generic pairing: proof of Theorem \ref{th:B}}\label{sec:pfB}

In Theorem \ref{th:A}, one can replace the logarithmic complex $\Omega^\cbbullet_Y(\log H)$ with the meromorphic complex $\Omega^\cbbullet_Y(*H)$ if one also replaces polynomials in $u$ with Laurent polynomials in $u$, and we will instead consider the $\CC[u,u^{-1}]$-module
\[
\bH^k(U,(\Omega^\cbbullet_U[u,u^{-1}],\rd+\rd f/u))
\]
with its connection induced by the action of $\partial_u$ coming from that of $\partial_u-f/u^2$ on each term of the complex.

\begin{lemma}\label{lem:inductive}
The restriction morphisms
\begin{multline}\tag{\ref{lem:inductive}$\,*$}\label{eq:inductive*}
\bH^k(X,(\Omega^\cbbullet_X(*D)[u,u^{-1}],\rd+\rd f/u))\\
\to\bH^k(Y,(\Omega^\cbbullet_Y(*H)[u,u^{-1}],\rd+\rd f/u))\\
\to\bH^k(U,(\Omega^\cbbullet_U[u,u^{-1}],\rd+\rd f/u))
\end{multline}
are isomorphisms of $\CC[u,u^{-1}]$-modules compatible with the action of $\partial_u$ for each $k$, as well as the natural morphisms
\begin{multline}\tag{\ref{lem:inductive}$\,**$}\label{eq:inductive**}
\bH^k\bigl(Y,(u^{-\cbbullet}\Omega^\cbbullet_Y(\log H)[u],\rd+\rd f/u)\bigr)\otimes_{\CC[u]}\CC[u,u^{-1}]\\
\to\bH^k\bigl(Y,(u^{-\cbbullet}\Omega^\cbbullet_Y(\log H)[u,u^{-1}],\rd+\rd f/u)\bigr)\\
\to\bH^k(Y,(\Omega^\cbbullet_Y(*H)[u,u^{-1}],\rd+\rd f/u)).
\end{multline}
\end{lemma}

\begin{proof}
Compatibility with the action of $\partial_u$ is clear, as it already holds at the level of complexes. The isomorphism property follows Lemma \ref{lem:commutinductive}, except for the last morphism of \eqref{eq:inductive**}.

For the latter, its left-hand side is $\CC[u,u^{-1}]$-free of finite rank, after Theorem \ref{th:A}. Its rank is given in Remark \ref{rem:computrk}.
The right-hand side is interpreted as the localized Fourier transform of the pushforward by $f$ (at a suitable degree) of the $\cD_Y$\nobreakdash-mod\-ule~$\cO_Y(*H)$. It is well-known to be $\CC[u,u^{-1}]$-free of the rank given by Remark \ref{rem:computrk}. If we fix $u_o\neq0$, then the corresponding morphism in cohomology is an isomorphism, as follows from \cite[Cor.\,1.4]{Yu12} and noticed in \cite[\S1.2]{E-S-Y13}. The conclusion follows from Lemma~\ref{lem:basicu} (over the ring $\CC[u,u^{-1}]$ instead of $\CC[u]$).\qed
\end{proof}

\begin{remark}\label{rem:Fourier}
The interpretation in terms of Fourier transform shows that the action of $\partial_u$ has a regular singularity at $u=\infty$, since the Gauss-Manin systems of $f$ have regular singularity at each of their singularities.
\end{remark}

In order to interpret $\bH^k\bigl(Y,(\Omega^\cbbullet_Y(\log H)(-H)[u,u^{-1}],\rd+\rd f/u)\bigr)$ (to which the first isomorphism of \eqref{eq:inductive**} applies in a similar way) in meromorphic terms, we work in the category of holonomic $\cD$-modules on $Y$. More precisely, we~consider the ring\enlargethispage{\baselineskip}
\[
\cD_Y^u:=\cD_Y[u,u^{-1}]=\cD_Y\otimes_{\CC}\CC[u,u^{-1}]
\]
of algebraic differential operator on $Y$ with coefficients in $\cO^u_Y:=\cO_Y[u,u^{-1}]$, so that the base ring is $\CC[u,u^{-1}]$ instead of the field $\CC$ (it would be equivalent to consider differential operators on $Y\times\Gm$ relative to the projection to $\Gm$).

The $\cD_Y^u$-module $E_Y^{f/u}=(\cO^u_Y,\rd+\rd f/u)$ comes with two localizations along~$H$, denoted by $E_Y^{f/u}(*H)$ and $E_Y^{f/u}(!H)$. They satisfy the following properties:
\begin{itemize}
\item
$E_Y^{f/u}(*H)$ is generated, as a $\cD_Y^u$-module, by $E_Y^{f/u}(H)$ which is a $\cD_Y^u(\log H)$-module.\footnote{$\cD_Y^u(\log H)$ is locally generated by vector fields which are logarithmic along $H$, and the notation $(-\log H)$ should be more adapted.} It satisfies moreover
\[
E_Y^{f/u}(*H)\simeq\cD_Y^u\otimes_{\cD_Y^u(\log H)}(E_Y^{f/u}(H)).
\]
\item
$E_Y^{f/u}(!H)$ is defined as
\[
E_Y^{f/u}(!H):=\cD_Y^u\otimes_{\cD_Y^u(\log H)}(E_Y^{f/u}(-H)),
\]
where $E_Y^{f/u}(-H)$ is regarded as a $\cD_Y^u(\log H)$-module.
\end{itemize}

\begin{lemma}
The $\CC[u,u^{-1}]$-modules $H^k_\dR(Y,E_Y^{f/u}(*H))$ and $H^k_\dR(Y,E_Y^{f/u}(!H))$ are free of finite rank for each $k$.
\end{lemma}

\begin{proof}
For a \emph{regular holonomic} $\cD$-module $M$ on the affine line~$\Afu$ with affine coordinate $t$, the de~Rham cohomology \hbox{$H^k_\dR(\Afu,M[u,u^{-1}]\otimes E^{t/u})$} is nonzero in degree $k=1$ at most, and this cohomology is $\CC[u,u^{-1}]$-locally free of finite rank (\cf\eg\cite{Malgrange91}). On~noting that, for $\star=*,!$, we have \hbox{$E_Y^{f/u}(\star H)\simeq\cO_Y(\star H)\otimes_{\cO_Y}E_Y^{f/u}$} (where $\cO_Y(\star H)$ is defined in a way similar to that of $E_Y^{f/u}(\star H)$), we obtain the statement by applying the previous result to the $\cD$-module pushforward \hbox{$M=\cH^{k-1-n}f_+\cO_Y(\star H)$}, which is known to be regular holonomic.\qed
\end{proof}

\begin{proof}[Proof of Theorem \ref{th:B}]
As already noticed, the right-hand side in \eqref{eq:inductive**} can be rewritten as $H^{k-n}_{\dR}(Y,E_Y^{f/u}(*H))$, and this yields the first line \eqref{eq:B*} of Theorem \ref{th:B}.

We now prove the isomorphism of the second line \eqref{eq:B!}. The proof can also be adapted to the first line, and thereby gives another way to obtain \eqref{eq:inductive**}. For a $\cD^u_Y$\nobreakdash-module $M$, the de~Rham complex is a realization (up to a shift $[n]$) of $\omega^u_Y\otimes_{\cD^u_Y}^L\nobreak M$ and, similarly, for a $\cD^u_Y(\log H)$-module $N$, the logarithmic de~Rham complex of $N$ is a realization (up to a shift $[n]$) of $\omega^u_Y\otimes_{\cD^u_Y(\log H)}^LN$, where the canonical sheaf $\omega^u_Y=\omega_Y\otimes_{\CC}\nobreak\CC[u,u^{-1}]$ is equipped with its natural structure of right $\cD^u_Y$\nobreakdash-module. We~interpret (up to a shift $[n]$) the left-hand side of \eqref{eq:B!} as the hypercohomology on~$Y$ of $\omega^u_Y\otimes_{\cD^u_Y(\log H)}^LE_Y^{f/u}(-H)$ and the right hand side as that of $\omega^u_Y\otimes_{\cD^u_Y}^LE_Y^{f/u}(!H)$. Due to the isomorphism
\[
\omega^u_Y\otimes_{\cD^u_Y}^L(\cD^u_Y\otimes_{\cD^u_Y(\log H)}^LE_Y^{f/u}(-H))\simeq\omega^u_Y\otimes_{\cD^u_Y(\log H)}^LE_Y^{f/u}(-H),
\]
the isomorphism \eqref{eq:B!} would follow from the isomorphism
\[
\cD^u_Y\otimes_{\cD^u_Y(\log H)}E_Y^{f/u}(-H)\simeq \cD^u_Y\otimes_{\cD^u_Y(\log H)}^LE_Y^{f/u}(-H).
\]
Although $\cD^u_Y$ is not $\cD^u_Y(\log H)$-flat, one can use the criterion of \cite[Prop.\,B.5]{E-S18}: this isomorphism holds if in any local coordinate system adapted to the divisor $H$ where $H=\{x_1\cdots x_\ell=0\}$, any subsequence of the sequence $(x_1,\dots,x_\ell)$ is a regular sequence for $E_Y^{f/u}(-H)$. In the present setting, this criterion is easily checked, and this ends the proof of \eqref{eq:B!}. By replacing $E_Y^{f/u}(-H)$ with $E_Y^{f/u}(H)$, one would obtain another proof of \eqref{eq:B*}.\qed
\end{proof}

\section{Another approach to Theorem \ref{th:A}}\label{sec:RF}

\subsection{Application of the theory of mixed Hodge module}
We will make use of the theory of mixed Hodge modules of M.\,Saito \cite{MSaito86,MSaito87} in order to prove $\CC[u]$-freeness in Theorem~\ref{th:A}.

We consider the algebraic mixed Hodge modules on $Y$ whose underlying filtered $\cD_Y$-modules are respectively $(\cO_Y(*H),F_\bbullet\cO_Y(*H))$ and $(\cO_Y(!H),F_\bbullet\cO_Y(!H))$. There exists a natural morphism of mixed Hodge modules inducing the natural morphism of filtered $\cD_Y$-modules:
\[
(\cO_Y(!H),F_\bbullet\cO_Y(!H))\to(\cO_Y(*H),F_\bbullet\cO_Y(*H)).
\]
It is understood that the Hodge filtrations $F_\bbullet\cO_Y(!H),F_\bbullet\cO_Y(*H)$ are coherent filtrations with respect to the filtration $F_\bbullet\cD_Y$ of $\cD_Y$ by the order of differential operators.

We consider the Rees module construction, by setting $R_F\cD_Y=\bigoplus_kF_k\cD_Yu^k$ and similarly for filtered $\cD_Y$-modules. In particular, $R_F\cO_Y=\cO_Y[u]$.

We set $G_0[\star H]=R_F(\cO_Y(\star H))$, with $\star={!},{*}$, that we consider as an $R_F\cD_Y$-module. We~obtain corresponding de~Rham complexes of $\CC[u]$-modules:
\begin{multline*}
\DR_YG_0[\star H]=\bigl\{0\ra R_F(\cO_Y(\star H))\xrightarrow{\textstyle\rd+\rd f/u}\cdots\\
\xrightarrow{\textstyle\rd+\rd f/u}u^{-n}\Omega^n_Y\otimes R_F(\cO_Y(\star H))\ra0\bigr\}.
\end{multline*}

\begin{proposition}\label{prop:RF}
We have a commutative diagram for each $k$:
\[
\xymatrix@C=.4cm{
\bH^k\bigl(Y,(u^{-\cbbullet}\Omega_Y^\cbbullet(\log H)(-H)[u],\rd+\rd f/u)\bigr)\ar[d]_\wr\ar[r]&\bH^k\bigl(Y,(u^{-\cbbullet}\Omega_Y^\cbbullet(\log H)[u],\rd+\rd f/u)\bigr)\ar[d]^\wr\\
\bH^k\bigl(Y,\DR_YG_0[!H]\bigr)\ar[r]&\bH^k\bigl(Y,\DR_YG_0[*H]\bigr)
}
\]
\end{proposition}

\begin{proof}
The proof of the isomorphisms is similar to that of Theorem \ref{th:B} by using the sheaves $R_F\cD_Y$ and $R_F\cD_Y(\log H)$ (deno\-ted respectively $\cswt\cD_Y$ and $\cswt\cD_Y(\log H)$ in \cite{E-S18}) instead of the sheaves~$\cD_Y^u$ and $\cD_Y^u(\log H)$. The criterion of \cite[Prop.\,B.5]{E-S18} applies in a straightforward way.\qed
\end{proof}

\begin{remark}\label{rem:finitetype}
The notation $G_0(*H)$ (instead of $G_0[*H]$) would mean $(R_F\cO_Y)(*H)$ (instead of $R_F(\cO_Y(*H))$). Thus $G_0(*H)=\cO_Y(*H)[u]$, and the associated twisted de~Rham complex is that considered at the end of Section \ref{subsec:nota}. Its hypercohomology may not be of finite type, as we have seen in \ref{subsec:comparing}\eqref{enum:a}. This means that, in order to obtain finite type, one needs to take into account the Hodge filtration at infinity on~$U$ (with $f$ remaining finite). A similar comment applies to $G_0(!H)$. On~the other hand, Theorem \ref{th:B} shows that this distinction disappears if we invert $u$.
\end{remark}

\subsection{The case of tame functions}
For tame functions, we will show that there is no need to twist by $(-H)$ in the pairing \eqref{eq:pairY} of Theorem~\ref{th:A}, similarly to what occurs in the proper case.

\begin{definition}[Katz tameness, {\cf\cite[Prop.\,14.13.3]{Katz90}}]\label{def:tame}
We say that $f:U\to\CC$ is Katz-tame if the cone of the natural morphism of complexes
\[
\bR f_!\CC_U\to\bR f_*\CC_U
\]
has constant cohomology sheaves.
\end{definition}

\begin{proposition}\label{prop:tame}
Assume that $f:U\to\CC$ is Katz-tame. Then, for every $k$, the natural morphism
\begin{equation}\tag{\ref{prop:tame}$\,*$}\label{eq:tame}
\bH^k\bigl(Y,\DR_YG_0[!H]\bigr)\to\bH^k\bigl(Y,\DR_YG_0[*H]\bigr)
\end{equation}
is an isomorphism.
\end{proposition}

\begin{proof}
If $M$ is a left $\cD_{\Afu}$-module on the affine line with coordinate $t$, we also consider it as an $\cO_{\Afu}$-module with connection $\nabla$. Let $(M,F_\bbullet M)$ be the coherent filtered $\cD$-module and let $R_FM$ denote the associated $R_F\cD_{\Afu}$-module equipped with its $u$\nobreakdash-connection $u\nabla$. The twisted de~Rham complex
\[
0\to R_FM\To{\nabla+\rd t/u}u^{-1}\Omega^1_{\Afu}\otimes R_FM\to0
\]
has nonzero cohomology in degree $1$ at most (\ie $\nabla+\rd t/u$ is injective). Furthermore, it also has nonzero hypercohomology in degree $1$ at most: hypercohomology is computed by means of global sections on $\Afu$ of the latter complex, and injectivity of $\nabla+\rd t/u$ is checked similarly.

On the other hand, for each $j$, the $j$-th pushforward by the map $f:Y\to\Afu$ of the Rees modules $R_F\cO_Y(\star H)$ ($\star={!},{*}$) takes the form $R_FM$ for some filtered $\cD_{\Afu}$\nobreakdash-mod\-ule~$M$: this property is equivalent to the degeneration at $E_1$ of the spectral sequence attached to the proper pushforward of a coherent filtered $\cD_Y$-module when the latter underlies a mixed Hodge module \cite{MSaito87}. It follows that $\bH^k\bigl(Y,\DR_YG_0[\star H]\bigr)$ can be computed as the hypercohomology of $R_FM_\star^k$, where $R_FM_\star^k$ is the $R_F\cD_{\Afu}$-module underlying the pushforward (of the suitable degree) of $R_F\cO_Y(\star H)$.

The morphism of constructible complexes in the assumption of Katz-tameness comes from a morphism of the corresponding objects in the derived category $\catD^\rb(\MHM(\Afu))$ (\cf\cite{MSaito87}), hence of the corresponding cohomology mixed Hodge modules $\cM_!^j$ and $\cM_*^j$.

\begin{lemma}
Under the assumption of Katz-tameness, for each $k$, the kernel and cokernel of the natural morphism $\cM_!^j\to\cM_*^j$ are constant mixed Hodge modules (\ie whose associated perverse sheaf is the constant sheaf up to a shift).
\end{lemma}

\begin{proof}
It is a matter of proving that a constructible complex on $\Afuan$ whose cohomology is constant has also constant perverse cohomology. This is standard (\eg by using that a constructible complex on $\Afuan$ has constant cohomology \resp perverse cohomology if and only if for any $c\in\Afuan$ the associated complex of vanishing cycles at $c$ is isomorphic to zero).\qed
\end{proof}

A constant mixed Hodge module on $\Afu$ has a finite filtration (the weight filtration) whose pure graded Hodge modules are also constant, and the associated filtered $\cD_{\Afu}$\nobreakdash-modules are isomorphic to $\cO_{\Afu}$ with its standard filtration possibly shifted, so that $R_F\cO_{\Afu}\simeq\cO_{\Afu}[u]$. With the notation above, for each $k$, the kernel and cokernel of the natural morphism $R_FM_!^k\to R_FM_*^k$ are of that form.

By the first part of the proof, that the natural morphism \eqref{eq:tame} is an isomorphism will thus be proved if we prove that the hypercohomology of the twisted de~Rham complex associated to $R_F\cO_{\Afu}$ is zero, that is, the kernel and cokernel of
\[
\CC[t,u]\To{\partial_t+1/u}u^{-1}\CC[t,u]
\]
are zero. This is a simple check (for the vanishing of the cokernel, one uses that a polynomial of degree $d$ in $\CC[t]$ is annihilated by $\partial_t^{d+1}$).\qed
\end{proof}

It follows immediately from Proposition \ref{prop:RF} that, in the tame case, we can omit the twist by $(-H)$ in \eqref{eq:pairY}:

\begin{corollary}\label{cor:tame}
Assume that $f:U\!\to\!\CC$ is Katz-tame. Then, for each $k$, the pairing~\eqref{eq:pairY} induces a nondegenerate pairing between free $\CC[u]$-modules of finite rank:
\begin{multline*}
\bH^{n+k}\bigl(Y,(u^{-\cbbullet}\Omega_Y^\cbbullet(\log H)[u],\rd+\rd f/u)\bigr)\\
\otimes_{\CC[u]} \bH^{n-k}\bigl(Y,(u^{-\cbbullet}\Omega_Y^\cbbullet(\log H)[u],\rd-\rd f/u)\bigr)
\to u^{-n}\CC[u].
\end{multline*}
\end{corollary}

According to \eqref{eq:inductive**}, this implies:

\begin{corollary}\label{cor:tameb}
Assume that $f:U\!\to\!\CC$ is Katz-tame. Then, for each $k$, the pairing~\eqref{eq:pairY} induces a nondegenerate pairing between free $\CC[u,u^{-1}]$-modules of finite rank:
\begin{multline*}
\bH^{n+k}(U,(\Omega^\cbbullet_U[u,u^{-1}],\rd+\rd f/u))\\\otimes_{\CC[u,u^{-1}]} \bH^{n-k}(U,(\Omega^\cbbullet_U[u,u^{-1}],\rd+\rd f/u))
\to\CC[u,u^{-1}].
\end{multline*}
\end{corollary}

\begin{example}\label{ex:tame}
Assume that $H$ is smooth and $f_{|H}$ is smooth. It follows that for each $t\in\Afu$, $f^{-1}(t)$ is smooth near $H$ and cuts $H$ transversally. Let $j:U\hto Y$ denote the open inclusion and $i:H\hto Y$ the complementary closed inclusion. Then the cone of the natural morphism $j_!\CC_U\to \bR j_*\CC_U$ is supported on $H$ and its restriction to~$H$ is isomorphic to $i^{-1}\bR j_*\CC_U$, which has constant cohomology sheaves. It follows that the Katz-tameness condition is satisfied by $f$ on $U$ (use \eg the argument with vanishing cycles and the good behaviour of the vanishing cycle functor by proper pushforward). Let us also notice that the assumptions in Theorem \ref{th:Ch} are also fulfilled in this example.
\end{example}

\section{The formal pairing}\label{sec:pfC}

\subsection{Proof of Theorem \texorpdfstring{\ref{th:Ah}}{refAh}}
We~first make formal the Kontsevich complex with the $u$-parameter. For that purpose, we set $\Omega_f^k\lcr u\rcr:=\varprojlim_\ell\Omega_f^k[u]/u^\ell\Omega_f^k[u]$. We~refer to \cite{Bibi10b} for some properties of this construction.

\begin{proposition}\label{prop:formalization}
For any $k$, we have
\[
\bH^k\bigl(X,(\Omega^\cbbullet_f\lcr u\rcr,u\rd+\rd f)\bigr)=\CC\lcr u\rcr\otimes_{\CC[u]}\bH^k\bigl(X,(\Omega^\cbbullet_f[u],u\rd+\rd f)\bigr)
\]
and a similar property for $\Omega^\cbbullet_f(-D)$, $\Omega^\cbbullet_f(-H)$, and the pairs $(Y,\Omega^\cbbullet_Y(\log H))$ and $(Y,\Omega^\cbbullet_Y(\log H)(-H))$.
\end{proposition}

\begin{proof}
We prove the case of $\Omega_f^\cbbullet$, the other cases being proved similarly (by using Lemma \ref{lem:XY} for the latter two pairs). By a straightforward induction, arguing as in Lemma \ref{lem:basicu} (due to the freeness property in Corollary \ref{cor:dualiteuddf}), we find for each $\ell\geq1$:
\begin{multline*}
\bH^k\bigl(X,(\Omega^\cbbullet_f[u]/u^\ell\Omega^\cbbullet_f[u],u\rd+\rd f)\bigr)\\
\simeq\bH^k\bigl(X,(\Omega^\cbbullet_f[u],u\rd+\rd f)\bigr)/u^\ell\bH^k\bigl(X,(\Omega^\cbbullet_f[u],u\rd+\rd f)\bigr),
\end{multline*}
so that the conclusion follows.\qed
\end{proof}

\begin{proof}[of Theorem \ref{th:Ah}]
We just apply the functor $\CC\lcr u\rcr\otimes_{\CC[u]}$ to the statements of Theorem \ref{th:A}, according to Proposition \ref{prop:formalization}.\qed
\end{proof}

\subsection{Proof of Theorem \texorpdfstring{\ref{th:Ch}}{Ch}}\label{sec:formalmeromorphic}
We consider the complex $(\Omega_U^\cbbullet\lcr u\rcr,u\rd+\rd f)$. The $\CC$\nobreakdash-cons\-tructible complex of vanishing cycles $\phi_{f-c}\CC_{U^\an}$ (which is perverse up to a shift) will come into play.

\begin{proposition}\label{prop:C}
Assume that the critical set of $f$ is compact. Then the $\CC\lcr u\rcr$-module $\bH^k\bigl(U,(\Omega_U^\cbbullet\lcr u\rcr,u\rd+\rd f)\bigr)$ is of finite type for each~$k$. It is $\CC\lcr u\rcr$\nobreakdash-free for every $k$ if and only if
\[
\dim\bH^k\bigl(U,(\Omega_U^\cbbullet,\rd f)\bigr)=\sum_{c\in\CC}\dim\bH^{k-1}(f^{-1}(c)^\an,\phi_{f-c}\CC_{U^\an})\quad\forall k.
\]
\end{proposition}

Note that we do not assert the existence of an isomorphism similar to that of Proposition \ref{prop:formalization} for $\bH^k\bigl(U,(\Omega_U^\cbbullet\lcr u\rcr,u\rd+\rd f)\bigr)$, since we do not know whether the assumption of finite-dimensionality of $\bH^k\bigl(U,(\Omega_U^\cbbullet,\rd f)\bigr)$ is enough to ensure that $\bH^k\bigl(U,(\Omega_U^\cbbullet\lcr u\rcr,u\rd+\rd f)\bigr)$ is of finite type over $\CC[u]$. We~first review some results of \cite{B-S07,Bibi10b,S-MS12}.

\begin{lemma}\label{lem:restr}
For each $k$, the natural restriction morphism
\[
\bH^k\bigl(X,(\Omega_X^\cbbullet(*D)\lcr u\rcr,u\rd+\rd f)\bigr)\to\bH^k\bigl(U,(\Omega_U^\cbbullet\lcr u\rcr,u\rd+\rd f)\bigr)
\]
is an isomorphism.
\end{lemma}

\begin{proof}
By a spectral sequence argument, it is enough to prove that for any $j,k$, the restriction morphism
\[
H^k\bigl(X,\Omega_X^j(*D)\lcr u\rcr\bigr)\to H^k\bigl(U,\Omega_U^j\lcr u\rcr\bigr)
\]
is an isomorphism. It follows from \cite[\S4]{Hartshorne75}, as noted in \cite[\S2.a]{Bibi10b}, that in such cases, $\lcr u\rcr$ commutes with taking cohomology. Then the assertion is clear.\qed
\end{proof}

\begin{lemma}[Algebraic/analytic comparison]\label{lem:GAGA}
For each $k$, the natural morphism
\[
\bH^k\bigl(X,(\Omega_X^\cbbullet(*D)\lcr u\rcr,u\rd+\rd f)\bigr)\to\bH^k\bigl(X^\an,(\Omega_{X^\an}^\cbbullet(*D)\lcr u\rcr,u\rd+\rd f)\bigr)
\]
is an isomorphism.
\end{lemma}

\begin{proof}
We consider the morphism between the spectral sequences associated to the filtration of the de~Rham complexes by the stupid truncation. The morphism at the~$E_1$ level is
\[
\bH^k\bigl(X,\Omega_X^j(*D)\lcr u\rcr\bigr)\to\bH^k\bigl(X^\an,\Omega_{X^\an}^j(*D)\lcr u\rcr\bigr).
\]
By \cite[(2.3)]{Bibi10b}, this morphism is an isomorphism, which implies the lemma by a spectral sequence argument.\qed
\end{proof}

\begin{lemma}\label{lem:Rj*}
The natural morphism of complexes
\[
(\Omega_{X^\an}^\cbbullet(*D)\lcr u\rcr,u\rd+\rd f)\to\bR j_*(\Omega_{U^\an}^\cbbullet\lcr u\rcr,u\rd+\rd f)
\]
is a quasi-isomorphism.
\end{lemma}

\begin{proof}
Although the result of \cite[Prop.\,4.1]{Bibi10b} is stated for $\lpr u\rpr$ instead of $\lcr u\rcr$, one can follow its proof with $\lcr u\rcr$ instead of $\lpr u\rpr$.\qed
\end{proof}

\begin{lemma}\label{lem:BS}
The order of the $u$-torsion of each cohomology sheaf
\[
\cH^j(\Omega_{U^\an}^\cbbullet\lcr u\rcr,u\rd\!+\!\rd f)
\]
is locally bounded and, modulo its torsion, this sheaf is a constructible sheaf of $\CC\lcr u\rcr$\nobreakdash-modules of finite type.
\end{lemma}

\begin{proof}
This is essentially \cite[Th.\,1]{B-S07} if we work with the isomorphic complex $(u^{-\cbbullet}\Omega_U^\cbbullet[u],\rd+\rd f/u)$. Let us give details on the reduction to \loccit\ In \cite[(1.5.2)]{S-MS12}, two complexes are considered, with the notation $\partial_t^{-1}$ for our notation $u$ (and the function $-f$ is considered, instead of $f$). One is $F_0=(u^{-\cbbullet}\Omega_U^\cbbullet[u],\rd+\rd f/u)$ and for each $\ell\geq0$, the subcomplex $F_\ell$ with terms $u^\ell(u^{-\cbbullet}\Omega_U^\cbbullet[u])$. The other one is $G_0$ and similarly $G_\ell$, with inclusions $u^{\ell+1}F_0\subset u^\ell G_0\subset u^\ell F_0$ for each $\ell\geq0$, so that $F_0/G_0$ has $u$-torsion of order one. The result of \cite[Th.\,1]{B-S07} together with the identification of \cite[(1.5.5)]{S-MS12} implies the assertion of the lemma for the complex $\varprojlim_\ell(G_0/G_\ell)$. On~the other hand, the complex occurring in the lemma is $\varprojlim_\ell(F_0/F_\ell)$, so the inclusions above yield the statement of the lemma.\qed
\end{proof}

\begin{proof}[of Proposition \ref{prop:C}]
We start with the finiteness statement. According to Lemmas \ref{lem:restr}--\ref{lem:Rj*}, we are reduced to proving the $\CC\lcr u\rcr$-finiteness of \hbox{$\bH^k\bigl(U^\an,(\Omega_{U^\an}^\cbbullet\lcr u\rcr,u\rd+\rd f)\bigr)$} for each~$k$. 

The complex $(\Omega_{U^\an}^\cbbullet\lcr u\rcr,u\rd+\rd f)$ is supported on the critical set of \hbox{$f:U\to\Afu$}: indeed, this follows from the fact that, on the product $\Delta\times V^\an$ of an open disc $\Delta$ with coordinate $x$ by a complex manifold $V^\an$, the morphism
\[
\cO(\Delta\times V^\an)\lcr u\rcr\To{u\partial_x+1}\cO(\Delta\times V^\an)\lcr u\rcr
\]
is an isomorphism, which is easily checked. Since the critical set of $f$ is compact by assumption, we conclude from Lemma \ref{lem:BS} that the order of the $u$\nobreakdash-tor\-sion of each cohomology sheaf is bounded, say by $N$.

We consider the spectral sequence with
\[
E_2^{p,q}=H^p\bigl(U^\an,\cH^q(\Omega_{U^\an}^\cbbullet\lcr u\rcr,u\rd+\rd f)\bigr).
\]
Since it can be realized as the spectral sequence of a bounded double complex, it~converges at a finite step. Let $\cT^q$ be the $\CC\lcr u\rcr$-torsion of $\cH^q:=\cH^q(\Omega_{U^\an}^\cbbullet\lcr u\rcr,u\rd+\rd f)$ and $\cI^q$ the quotient $\cH^q/\cT^q$. Since $\cH^q$ is supported on a compact set, Lemma \ref{lem:BS} implies that $H^p(U^\an,\cI^q)$ has finite type over $\CC\lcr u\rcr$ and $H^p(U^\an,\cT^q)$ is of $N$\nobreakdash-torsion. Therefore, $E_2^{p,q}$ has $\CC\lcr u\rcr$-torsion of order bounded by $N$ and its quotient by torsion has finite type over $\CC\lcr u\rcr$. This property goes through the spectral sequence, and we conclude that it holds for $\bH^k\bigl(U^\an,(\Omega_{U^\an}^\cbbullet\lcr u\rcr,u\rd+\rd f)\bigr)$.

On the other hand, the quotient complex $(\Omega_{U^\an}^\cbbullet\lcr u\rcr/u^N\Omega_{U^\an}^\cbbullet\lcr u\rcr,u\rd+\rd f)$ has a finite filtration $F^\cbbullet$ whose graded terms are all isomorphic to the complex $(\Omega_{U^\an}^\cbbullet,\rd f)$. The latter is a complex in $\Mod_{\coh}(\cO_{U^\an})$ supported on the compact critical set of~$f$. Therefore, $\bH^k\bigl(U^\an,\gr_F(\Omega_{U^\an}^\cbbullet\lcr u\rcr/u^N\Omega_{U^\an}^\cbbullet\lcr u\rcr,u\rd+\rd f)\bigr)$ is finite dimensional, and~so~is $\bH^k\bigl(U^\an,(\Omega_{U^\an}^\cbbullet\lcr u\rcr/u^N\Omega_{U^\an}^\cbbullet\lcr u\rcr,u\rd+\rd f)\bigr)$. We~conclude the proof of the first statement by considering the hypercohomology exact sequence deduced from the exact sequence
\begin{multline*}
0\to(\Omega_{U^\an}^\cbbullet\lcr u\rcr,u\rd+\rd f)\xrightarrow{~\dpl u^N}(\Omega_{U^\an}^\cbbullet\lcr u\rcr,u\rd+\rd f)\\\to(\Omega_{U^\an}^\cbbullet\lcr u\rcr/u^N\Omega_{U^\an}^\cbbullet\lcr u\rcr,u\rd+\rd f)\to0.
\end{multline*}

We now consider the second statement. Since $\CC\lpr u\rpr$ is $\CC\lcr u\rcr$-flat, tensoring with $\CC\lpr u\rpr$ commutes with taking cohomology, and we have
\begin{align*}
\bH^k\bigl(U,(\Omega_U^\cbbullet\lpr u\rpr,u\rd+\rd f)\bigr)&=\CC\lpr u\rpr\otimes_{\CC\lcr u\rcr}\bH^k\bigl(U,(\Omega_U^\cbbullet\lcr u\rcr,u\rd+\rd f)\bigr)\\
&=\bH^k\bigl(U,(\Omega_U^\cbbullet\lcr u\rcr,u\rd+\rd f)\bigr)[u^{-1}].
\end{align*}
Indeed, this is seen first for each $H^k(U,\Omega_U^j\lpr u\rpr)$ by Noetherianity of $U$, and then deduced for $\bH^k\bigl(U,(\Omega_U^\cbbullet\lpr u\rpr,u\rd+\rd f)\bigr)$ by a spectral sequence argument already used.

We now apply the property that, for a $\CC\lcr u\rcr$-module of finite type $M$, $M$ is $\CC\lcr u\rcr$-free if and only if
\[
\dim_{\CC} M/uM=\dim_{\CC\lpr u\rpr}M[u^{-1}].
\]
Set $M^k=\bH^k\bigl(U,(\Omega_U^\cbbullet\lcr u\rcr,u\rd+\rd f)\bigr)$. We~argue by induction on the length of the long exact sequence of hypercohomology associated with the short exact sequence
\[
0\to(\Omega_U^\cbbullet\lcr u\rcr,u\rd+\rd f)\To{u}(\Omega_U^\cbbullet\lcr u\rcr,u\rd+\rd f)\to(\Omega_U^\cbbullet,\rd f)\to0.
\]

Let $k$ be such that $M^j=0$ for $j>k$. Then $M^k/uM^k=\bH^k\bigl(U,(\Omega_U^\cbbullet,\rd f)\bigr)$. On~the other hand, $\dim_{\CC\lpr u\rpr}M^k[u^{-1}]=\dim\sum_{c\in\CC}\dim\bH^{k-1}(f^{-1}(c)^\an,\phi_{f-c}\CC_{U^\an})$ by the main theorem of \cite{S-MS12} (\cf also \cite[Th.\,1.1]{Bibi10b}). We~conclude that $M^k$ is $\CC\lcr u\rcr$-free if and only if both dimensions are equal. In such a case, $u:M^k\to M^k$ is injective. We~can thus truncate the long exact sequence mentioned above after $k-1$ and we conclude by decreasing induction on $k$.\qed
\end{proof}

\begin{proof}[of the first part of Theorem \ref{th:Ch}]
In view of Proposition \ref{prop:C}, we only need to show the equality of dimensions occurring in that proposition under Condition \eqref{cond:b}. This is precisely \cite[Th.\,2]{Bibi97b}.\qed
\end{proof}

\begin{proof}[of the second part of Theorem \ref{th:Ch}]
In view of Theorem \ref{th:Ah}, it is enough to prove that the natural morphisms
\begin{multline*}
\bH^k\bigl(X,(\Omega^\cbbullet_f(-D)\lcr u\rcr,u\rd+\rd f)\bigr)\to\bH^k\bigl(X,(\Omega^\cbbullet_f\lcr u\rcr,u\rd+\rd f)\bigr)\\\to\bH^k\bigl(X,(\Omega_X^\cbbullet(*D)\lcr u\rcr,u\rd+\rd f)\bigr)
\end{multline*}
are isomorphisms. Together with Lemma \ref{lem:restr}, this shows the correspondence with the $\CC\lcr u\rcr$-modules of Theorem \ref{th:Ah}. These are free $\CC\lcr u\rcr$-modules of finite rank, according to the first part of Theorem \ref{th:Ch}. It is thus enough to prove this modulo $u\CC\lcr u\rcr$. We~are left with the morphisms
\[
\bH^k\bigl(X,(\Omega^\cbbullet_f(-D),\rd f)\bigr)\to\bH^k\bigl(X,(\Omega^\cbbullet_f,\rd f)\bigr)\to\bH^k\bigl(X,(\Omega_X^\cbbullet(*D),\rd f)\bigr).
\]
Since $X$ is compact, we can replace the complexes which one takes hypercohomology of by their analytic counterpart on $X^\an$ by GAGA. These analytic complexes are supported on the critical set of $f$, which is contained in $U$, hence they coincide in a neighbourhood of this set. The assertion follows.\qed
\end{proof}

\section{Geometry}\label{sec:geometry}

Let $\Afnp$ be the affine chart with coordinates $(x_0,\dots,x_n)$ in $\PP^{n+1}$ with complement~$\PP^n_\infty$, and let $\varpi:\cswt\PP^{n+1}\to\PP^{n+1}$ be the blow-up of $\PP^{n+1}$ at the origin, with exceptional divisor $\varpi^{-1}(0)=\PP^n$. We~identify $\varpi^{-1}(\Afnp)$ with the total space $\Tot(\cO_{\PP^n}(-1))$ and $\cswt\PP^{n+1}$ to $\PP(\cO_{\PP^n}(-1)\oplus\bun)$.

For $d\geq1$, we consider the action of $\mu_d$ on $\Afnp$ by $x\mto\zeta x$ ($\zeta\in\mu_d$). This action lifts to $\cswt\PP^{n+1}$. We~note that $\mu_d$ acts trivially on $\varpi^{-1}(0)$ and $\PP^n_\infty$. Moreover,
\[
\Tot(\cO_{\PP^n}(-1))/\mu_d\simeq\Tot(\cO_{\PP^n}(-d)),\qquad\PP(\cO_{\PP^n}(-1)\oplus\bun)/\mu_d\simeq \PP(\cO_{\PP^n}(-d)\oplus\bun).
\]
The quotient space $\PP^{n+1}/\mu_d$ is smooth away from the origin. We~have a commutative diagram
\[
\xymatrix{
\PP(\cO_{\PP^n}(-1)\oplus\bun)\ar[r]^-\varpi\ar[d]_{\cswt\rho_d}&\PP^{n+1}\ar[d]^{\rho_d}\\
\PP(\cO_{\PP^n}(-d)\oplus\bun)\ar[r]^-{\varpi_d}&\PP^{n+1}/\mu_d
}
\]

Let $f\in\CC[x_0,\dots,x_n]$ be a homogeneous polynomial of degree $d$, such that $f^{-1}(0)$ has an isolated singularity at the origin. Let $V\subset\PP^n$ denote the smooth hypersurface defined by $f$ and let $X$ denote the closure of the graph $\{t-\nobreak f(x)=\nobreak0\}\subset \CC^{n+1}\times\Afu_t$ in $\PP^{n+1}\times\Afu_t$. Since $f$ is invariant by $\mu_d$, it descends as a regular function $f_d$ on $\CC^{n+1}/\mu_d$ whose graph in \hbox{$(\CC^{n+1}/\mu_d)\times\Afu_t$} has closure \hbox{$X_d:=X/\mu_d\subset(\PP^{n+1}/\mu_d)\times\Afu_t$}. Similarly $\cswt f:=f\circ\varpi$ descends as a function on $\Tot(\cO_{\PP^n}(-d))$ whose graph has closure $\cswt X_d$ in $\PP(\cO_{\PP^n}(-d)\oplus\bun)\times\Afu_t$. There is a natural proper modification which is an isomorphism away from the origin:
\[
\pi_d:\cswt X_d\to X_d.
\]

\begin{lemma}
The space $X_d$ is smooth away from the origin, and the projection \hbox{$p:X_d\to\Afu_t$} is smooth away from the origin.\qed
\end{lemma}

As a consequence, the composition $\cswt f_d:\cswt X_d\to\Afu_t$ is smooth away from $\pi_d^{-1}(0)\simeq\nobreak\PP^n$.

\begin{lemma}
The critical fiber $\cswt f_d^{-1}(0)$ is a reduced divisor with two components, one being $\pi_d^{-1}(0)\simeq\PP^n$, intersecting normally along the smooth projective variety \hbox{$V\subset\pi_d^{-1}(0)$}.\qed
\end{lemma}

As a consequence, the vanishing cycle complex $\phi_{\cswt f_d}\CC_{\cswt X_d}$ is a complex of sheaves supported on $V$ and has cohomology in degree $1$ only. Moreover, $\cH^1\phi_{\cswt f_d}\CC_{\cswt X_d}$ is a local system of rank one, which is constant if $n\geq3$, since $V$ is $1$-connected.

We apply the results of Section \ref{sec:Kontsevich} to $Y=\cswt X_d$ and $\cswt f_d:\cswt X_d\to\Afu_t$. By Theorem~\ref{th:A}, $H^k(\cswt X_d,(\Omega_{\cswt X_d}^\cbbullet[u],u\rd+\rd\cswt f_d))$ is $\CC[u]$-free of finite rank. Its rank is given by the computation of Remark \ref{rem:computrk}, that is,
\[
\dim \bH^{k-1}(V,\phi_{\cswt f_d}\CC_{\cswt X_d})=\dim \bH^{k-1}(V,\cH^1\phi_{\cswt f_d}\CC_{\cswt X_d}[-1])=\dim H^{k-2}(V,\CC).
\]

\begin{proposition}
The function $\cswt f_d:Y=\cswt X_d\to\Afu_t$ together with $U:=\Tot(\cO_{\PP^n}(-d))$ and $H=Y\moins U$ satisfies the tameness property of Example \ref{ex:tame}.
\end{proposition}

\begin{proof}
Since $\cswt X_d$ and $X_d$ are equal in the neighbourhood of $H$, it is enough to work with $X_d$ and $f_d$. Recall that $X$ denotes the closure of the graph of $f$ in $\PP^{n+1}\times\Afu_t$. Then $X\moins(\Afnp\times\Afu_t)=X\cap(\PP^n_\infty\times\Afu_t)$ is the product $V\times\Afu_t$ and the restriction of~$f$ to it is simply the projection. It is thus obviously smooth, and the same property remains true after taking the quotient by $\mu_d$, since $\mu_d$ acts as the identity on $\PP^n_\infty$.\qed
\end{proof}

We conclude that Corollaries \ref{cor:tame}, \ref{cor:tameb} and Theorem \ref{th:Ch} apply to \hbox{$\cswt f_d:Y=\cswt X_d\to\Afu_t$}.

\subsubsection*{Acknowledgements}
This work grew out from discussions with Bumsig Kim during my visit at KIAS on February 2019. It was supposed to be a starting point to understanding some questions related to gauged linear sigma models like the one described in \cite[Ex.\,2.4]{F-K20}. It was Bumsig Kim who insisted to make precise the comparison between the various dualities occurring in this context. I would like to thank KIAS for the excellent working conditions during this visit. I thank Jeng-Daw Yu for explaining some parts of his article \cite{Yu12} and for his comments, and the referee for noticing a mistake in the original proof of Proposition \ref{prop:C} and providing a correction.

\bibliographystyle{spmpsci}
\bibliography{sabbah_duality-kim}

@PREAMBLE{ "\providecommand{\eprint}[1]{\href{http://arxiv.org/abs/#1}{\texttt{arXiv\string:\allowbreak#1}}}" }

@STRING{advam4	= "Adv. Math." }

@STRING{algeg	= "Algebraic Geom." }

@STRING{aster	= "Ast{\'e}risque" }

@STRING{forsig = "Forum Math. Sigma"}

@STRING{inshe	= "Publ. Math. Inst. Hautes {\'E}tudes Sci." }

@STRING{jalgg	= "J.~Algebraic Geom." }

@STRING{jdifg1 = "J.~Differential Geometry"}

@STRING{jlonm2	= "J.~London Math. Soc.~(2)" }

@STRING{jreia	= "J.~reine angew. Math." }

@STRING{manum	= "Manu\-scripta Math." }

@STRING{matha	= "Math. Ann." }

@STRING{micmj	= "Michigan Math.~J." }

@STRING{nagmj	= "Nagoya Math.~J." }

@STRING{rims	= "Publ. RIMS, Kyoto Univ." }

@STRING{sigma	= "SIGMA Symmetry Integrability Geom. Methods Appl." }

@STRING{tohmj	= "Tohoku Math.~J." }

@STRING{pm	= "Progress in Math." }

@STRING{docmat	= "Documents Math." }

@STRING{ams	= "American Mathematical Society" }

@STRING{smf	= "Soci{\'e}t{\'e} Math{\'e}matique de France" }

@Article{	  A-S97,
  author	= {Adolphson, A. and Sperber, S.},
  title		= {{On twisted de~Rham cohomology}},
  journal	= nagmj,
  volume	= {146},
  year		= {1997},
  pages		= {55--81}
}

@Article{	  Arapura97,
  author	= {Arapura, D.},
  title		= {Geometry of cohomology support loci for local systems.
		  {I}},
  journal	= jalgg,
  volume	= {6},
  year		= {1997},
  number	= {3},
  pages		= {563--597}
}

@InCollection{	  B-D04,
  author	= {Baldassarri, F. and D'Agnolo, A.},
  title		= {On {D}work cohomology and algebraic {$\mathcal
		  D$}\nobreakdash-mod\-ules},
  booktitle	= {Geometric aspects of {D}work theory},
  pages		= {245--253},
  publisher	= {Walter de Gruyter},
  address	= {Berlin},
  year		= {2004}
}

@Article{	  B-S07,
  author	= {Barlet, Daniel and Saito, M.},
  title		= {Brieskorn modules and {G}auss-{M}anin systems for
		  non-isolated hypersurface singularities},
  journal	= jlonm2,
  volume	= {76},
  year		= {2007},
  number	= {1},
  pages		= {211--224},
}

@Article{	  Brieskorn70,
  author	= {E. Brieskorn},
  title		= {{Die Monodromie der isolierten Singularit{\"a}ten von
		  Hyperfl{\"a}chen}},
  journal	= manum,
  volume	= {2},
  year		= {1970},
  pages		= {103-161}
}

@Article{	  C-Y16,
  author	= {Chen, K.-C. and Yu, J.-D.},
  title		= {{The K\"unneth formula for the twisted de Rham and Higgs
		  cohomologies}},
  journal	= sigma,
  volume	= {14},
  year		= {2018},
  note		= {article no.\,055, 14 p.}
}

@Article{	  D-M-S-S99,
  author	= {Dimca, A. and Maaref, Fay{\c{c}}al and Sabbah, C. and
		  Saito, M.},
  title		= {Dwork cohomology and algebraic $\mathcal{D}$-modules},
  journal	= matha,
  volume	= {318},
  year		= {2000},
  number	= {1},
  pages		= {107-125}
}

@Article{	  E-S-Y13,
  author	= {Esnault, H. and Sabbah, C. and Yu, J.-D.},
  title		= {{$E_1$-degeneration of the irregular Hodge filtration
		  (with an appendix by M.\,Saito)}},
  year		= {2017},
  journal	= jreia,
  volume	= {729},
  pages		= {171-227}
}

@Article{	  E-S18,
  author	= {Esnault, H. and Sabbah, C.},
  title		= {{Good lattices of algebraic connections}},
  year		= {2019},
  journal	= docmat,
  volume	= {24},
  pages		= {175-205}
}

@Unpublished{	  F-K20,
  author	= {Favero, D. and Kim, B.},
  title		= {{General GLSM invariants and their cohomological field
		  theories}},
  note		= {\eprint{2006.12182}},
  year		= {2020}
}

@Unpublished{	  Fan11,
  author	= {Fan, Huijun},
  title		= {Schr{\"o}dinger equations, deformation theory and
		  $tt^*$-geometry},
  year		= 2011,
  note		= {\eprint{1107.1290}}
}

@Article{	  Hartshorne75,
  author	= {R. Hartshorne},
  title		= {{On the de~Rham cohomology of algebraic varieties}},
  journal	= inshe,
  volume	= {45},
  year		= {1975},
  pages		= {5-99}
}

@Article{	  Hartshorne72,
  author	= {R. Hartshorne},
  title		= {Cohomology with compact supports for coherent sheaves on
		  an algebraic variety},
  journal	= matha,
  volume	= {195},
  year		= {1972},
  pages		= {199-207}
}

@InCollection{	  Hertling05,
  author	= {Hertling, C.},
  title		= {{Formes bilin{\'e}aires et hermitiennes pour des
		  singularit{\'e}s: un aper\-{\c c}u}},
  booktitle	= {{Singularit{\'e}s}},
  volume	= 18,
  publisher	= {Institut {\'E}lie Cartan},
  address	= {Nancy},
  year		= 2005,
  pages		= {1-17},
  note		= {English transl.: \eprint{2011.10099}}
}

@Book{		  Katz90,
  author	= {Katz, N.},
  title		= {Exponential sums and differential equations},
  series	= {Ann. of Math. studies},
  volume	= {124},
  publisher	= {Princeton University Press},
  address	= {Princeton, NJ},
  year		= {1990}
}

@Article{	  L-W22,
  author	= {Li, S. and Wen, H.},
  title		= {{On the $L^2$-Hodge theory of Landau-Ginzburg models}},
  journal	= advam4,
  volume	= {396},
  year		= {2022},
  note		= {article no.\,108165, 48\,p.}
}

@Book{		  Malgrange91,
  author	= {Malgrange, B.},
  title		= {{\'E}quations diff{\'e}rentielles {\`a} coefficients
		  polynomiaux},
  series	= pm,
  volume	= {96},
  publisher	= {Birkh{\"a}user},
  address	= {Basel, Boston},
  year		= {1991}
}

@InCollection{	  Pham85b,
  author	= {Pham, F.},
  title		= {{La descente des cols par les onglets de Lefschetz avec
		  vues sur Gauss-Manin}},
  editor	= {A. Galligo and J.-M. Granger and Maisonobe, {\relax Ph}.},
  booktitle	= {{Systèmes diff{\'e}rentiels et singularit{\'e}s (Luminy,
		  1983)}},
  series	= aster,
  publisher	= smf,
  volume	= {130},
  year		= {1985},
  pages		= {11-47}
}

@Article{	  S-MS12,
  author	= {Sabbah, C. and Saito, M.},
  title		= {{Kontsevich's conjecture on an algebraic formula for
		  vanishing cycles of local systems}},
  journal	= algeg,
  volume	= 1,
  number	= 1,
  year		= 2014,
  pages		= {107-130}
}

@Unpublished{	  Bibi13b,
  author	= {Sabbah, C.},
  title		= {{Vanishing cycles and their algebraic
		  computation}},
  year		= 2013,
  note		= {Lecture notes, Notre Dame,
		  \href{https://perso.pages.math.cnrs.fr/users/claude.sabbah/livres/sabbah_notredame1305.pdf}{\texttt{sabbah\_notredame1305.pdf}}}
}

@Unpublished{	  Bibi10b,
  author	= {Sabbah, C.},
  title		= {{On a twisted de Rham complex, II}},
  note		= {\eprint{1012.3818}},
  year		= 2010
}

@Article{	  Bibi97b,
  author	= {Sabbah, C.},
  title		= {{On a twisted de~Rham complex}},
  journal	= tohmj,
  volume	= {51},
  year		= {1999},
  pages		= {125-140}
}

@Article{	  S-Y14,
  author	= {Sabbah, C. and Yu, J.-D.},
  title		= {{On the irregular Hodge filtration of exponentially
		  twisted mixed Hodge modules}},
  year		= {2015},
  journal	= forsig,
  volume	= 3,
  note		= {article no.\,e9, 7 p.}
}

@InCollection{	  KSaito83,
  author	= {Saito, K.},
  title		= {The higher residue pairings {$K_F^{(k)}$} for a family of
		  hypersurfaces singular points},
  series	= {Proc. of Symposia in Pure Math.},
  booktitle	= {Singularities},
  publisher	= ams,
  volume	= {40},
  year		= {1983},
  pages		= {441-463}
}

@InCollection{	  S-T08,
  author	= {Saito, K. and Takahashi, A.},
  title		= {{From primitive forms to Frobenius manifolds}},
  booktitle	= {{From Hodge theory to integrability and TQFT:
		  tt*-geometry}},
  editor	= {Donagi, R. and Wendland, K.},
  series	= {Proc. Symposia in Pure Math.},
  publisher	= ams,
  address	= {Providence, R.I.},
  year		= 2008,
  volume	= 78,
  pages		= {31-48}
}

@Article{	  MSaito87,
  author	= {Saito, M.},
  title		= {{Mixed {Hodge} Modules}},
  journal	= rims,
  volume	= {26},
  year		= {1990},
  pages		= {221-333}
}

@Article{	  MSaito86,
  author	= {Saito, M.},
  title		= {Modules de {Hodge} polarisables},
  journal	= rims,
  volume	= {24},
  year		= {1988},
  pages		= {849-995}
}

@Article{	  Wei17,
  author	= {Wei, C.},
  title		= {Logarithmic comparison with smooth boundary divisor in
		  mixed {H}odge modules},
  journal	= micmj,
  volume	= 69,
  number	= 1,
  year		= 2020,
  pages		= {201--223}
}

@Article{	  Yu12,
  author	= {Yu, J.-D.},
  title		= {{Irregular Hodge filtration on twisted de Rham
		  cohomology}},
  year		= 2014,
  journal	= manum,
  volume	= 144,
  number	= {1--2},
  pages		= {99-133},
  eprint	= {1203.2338}
}

@Article{	  K-K-P14,
  author	= {Katzarkov, L. and Kontsevich, M. and Pantev, T.},
  title		= {{Bogomolov-Tian-Todorov theorems for Landau-Ginzburg
		  models}},
  year		= 2017,
  journal	= jdifg1,
  volume	= {105},
  number	= {1},
  pages		= {55-117}
}
\end{document}